\documentclass[12pt]{article}
\usepackage[utf8x]{inputenc}
\usepackage{amsmath,amsthm,amssymb}
\usepackage{mathtools} 
\usepackage{mathrsfs}
\usepackage[a4paper,left=2.5cm,right=2.5cm]{geometry}
\usepackage{verbatim}
\usepackage{hyperref}
\usepackage{xcolor}
\usepackage{enumitem}
\usepackage{csquotes}
\usepackage{authblk}

\theoremstyle{plain}
\newtheorem{conjecture}{Conjecture}
\newtheorem{lemma}{Lemma}
\newtheorem{proposition}{Proposition}
\newtheorem{theorem}{Theorem}
\newtheorem{corollary}{Corollary}
\theoremstyle{definition}
\newtheorem{definition}{Definition}
\newtheorem{algorithm}{Algorithm}
\theoremstyle{remark}
\newtheorem{remark}{Remark}
\setlist[enumerate,1]{label={\roman*)},ref={(\roman*)}} 
\setlist[enumerate,2]{label={\arabic*)},ref={(\arabic*)}} 

\newcommand{\R}{\mathbb{R}}  
\newcommand{\Z}{\mathbb{Z}}  
\newcommand{\N}{\mathbb{N}}  

\newcommand{\bfr}{{\bf r}}

\newcommand{\bfx}{{\bf x}}

\newcommand{\Id}{\operatorname{Id}}

\newcommand{\prox}{\operatorname{Prox}}
\newcommand{\Tr}{\operatorname{Tr}}
\newcommand{\dom}{\operatorname{dom}}
\newcommand{\ran}{\operatorname{ran}}

\newcommand{\Hilb}{\mathcal{H}}
\newcommand{\ext}{{\text{ext}}}

\newcommand{\supdiff}{\overline{\partial}}
\newcommand{\subdiff}{\underline{\partial}}

\newcommand{\hilite}[1]{}

\usepackage{braket}

\title{Moreau--Yosida regularization in  DFT}
\author{Simen Kvaal \\ \url{simen.kvaal@kjemi.uio.no}}
\affil{Hylleraas Centre for Quantum Molecular Sciences,\\
Department of Chemistry, University of Oslo,\\
NO-0315, Oslo, Norway}
\date{June 2022}

\begin{document}
\maketitle
\begin{displayquote}
    This is a preprint of the following chapter: S.~Kvaal, \emph{Moreau--Yosida Regularization in DFT}, to be published in  \emph{Density Functional Theory — Modeling, Mathematical Analysis, Computational Methods, and Applications}, edited by Eric Cancès and Gero Friesecke, Springer (2022). Reproduced with permission of Springer. 
\end{displayquote}

\begin{abstract}
Moreau--Yosida regularization is introduced into the framework of exact DFT. Moreau--Yosida regularization is a lossless operation on lower semicontinuous proper convex functions over separable Hilbert spaces, and when applied to the universal functional of exact DFT (appropriately restricted to a bounded domain), gives a reformulation of the ubiquitous $v$-representability problem and a rigorous and illuminating derivation of Kohn--Sham theory.

The chapter comprises a self-contained introduction to exact DFT, basic tools from convex analysis such as sub- and superdifferentiability and convex conjugation, as well as basic results on the Moreau--Yosida regularization. The regularization is then applied to exact DFT and Kohn--Sham theory, and a basic iteration scheme based in the Optimal Damping Algorithm is analyzed. In particular, its global convergence established. Some perspectives are offered near the end of the chapter.
\end{abstract}
\newpage
\tableofcontents


\section{Introduction}
\label{sec:introduction}

In this chapter, we introduce  Moreau--Yosida regularization into the exact DFT framework. Such regularization of convex optimization problems is lossless: Given a small $\epsilon>0$, a non-smooth convex functional $F[\rho]$ is regularized into a differentiable convex functional ${}^\epsilon F[\rho]$. However, this smoothing is invertible, such that the information encoded is not lost. Thus, the regularization gives a \emph{reformulation} of the convex optimization problem as a differentiable convex optimization problem. When applied to exact DFT, we gain insight into the ubiquitous $v$-representability problem and also a rigorous formulation of Kohn--Sham theory. This comes at a small price,  as we need to describe densities and potentials using Banach spaces that are reflexive, which excludes, say, exact Coulomb potentials on infinite domains. This trade-off is acceptable, as no physical system is truly infinite in extent. Indeed, placing the system in a finite but large box is the only approximation being made in this chapter.

The motivation for introducing Moreau--Yosida regularization comes from two directions: First, DFT is committed to treating ``all possible'' potentials at once. For the exact interacting problem, this is surely overkill: After all we are mostly interested in the atomic Coulomb potentials, or at most well-behaved potentials modeling, say, harmonic traps. On the other hand, passage to the non-interacting Kohn--Sham system turns the problem upside down: The density is the known quantity, while the effective potential is the unknown function of the density.  The standard approach is to differentiate the universal functional to obtain the exchange-correlation potential, but \emph{this is not permitted in exact DFT}. 
Indeed, the large class of potentials implies that the universal functional will be highly non-smooth.
%
We note in passing, that in the strongly correlated limit of DFT, where one obtains the purely-interacting universal functional, the exchange-correlation potential can in fact be obtained rigorously by differentiation, as shown in the chapter by Friesecke, Gerolin and Gori--Giorgi.
The second motivation for Moreau--Yosida regularization is that formal Kohn--Sham theory relies on densities to be both interacting and non-interacting $v$-representable. This problem vanishes in thin air in the regularized Kohn--Sham formulation.

While Moreau--Yosida regularization resolves some problems, it still leaves open the question of how to approximate the regularized universal functional ${}^\epsilon F[\rho]$ in a systematic manner. On the other hand, the method of Moreau--Yosida regularization is applicable also to model density functional approximations that are already convex.

This chapter has the following structure.
In Section~\ref{sec:exact-dft} we give a brief introduction to the convex formulation of exact DFT as given by Lieb~\cite{Lieb1983}. Some results from the theory of quadratic forms over Hilbert spaces are needed, and we will refer the reader to the excellent literature that exists. The article~\cite{Simon1971} by Simon lays out the formulation of quantum mechanics using quadratic forms, and the monograph~\cite{Schmuedgen2012} by Schm\"udgen gives further details, including the representation theorem of bounded-below closed quadratic forms. We include some definitions and theorems from the analysis of convex functions over Banach spaces, and also give proofs whenever deemed instructive. For an excellent and accessible introduction to convex analysis, see the short monograph~\cite{VanTiel1984} by van Tiel. After introducing exact DFT, we proceed to introduce the only approximation we need, that is, truncation of infinite space $\R^3$ to a bounded domain $\Omega \subset \R^3$. This allows formulation of exact DFT in a Hilbert space setting, or more generally in a reflexive Banach space setting.

In Section~\ref{sec:my} we introduce the Moreau--Yosida regularization of convex problems over Hilbert spaces. For a more detailed exposition for the Hilbert space setting, see the monograph~\cite{Bauschke2011} by Bauschke and Combettes. Most results can be wholly or partially generalized to convex functions over reflexive Banach spaces, see the monograph~\cite{Barbu2012} by Barbu and Precupanu. In our exposition, we include some important definitions and results, with complete proofs in many cases due to the central nature of the material. In the following Section~\ref{sec:my-dft}, we apply Moreau--Yosida regularization to box truncated exact DFT. In particular, we detail the rigorous derivation of Kohn--Sham theory in the regularized setting, including a basic analysis of the Moreau--Yosida Kohn--Sham optimal damping algorithm (MYKSODA). Our treatment is adapted from Refs.~\cite{Kvaal2014,Kvaal2015,Laestadius2018b,Penz2019,Penz2019erratum}.

Finally, in Section~\ref{sec:conclusion} we provide a conclusion and discuss opportunities for future research.

The author gratefully acknowledges helpful feedback from Andre Laestadius and Markus Penz, and also Gero Friesecke who additionally provided the complete proof of the fact that finite-kinetic energy $N$-electron wavefunctions have densities with square roots in the Sobolev space $H^{1}(\R^3)$, see Theorem~\ref{thm:N-rep}.

\section{Exact DFT}
\label{sec:exact-dft}

In this section, we give a brief outline of elements of exact DFT, paying special attention to the regularity, or lack thereof, of the universal functional and ground-state energy. We then regularize exact DFT to a finite but large subdomain $\Omega \subset \R^3$. Our focus is on molecular electronic systems.

\subsection{The variational ground-state problem}
\label{sec:exact-dft-1}
In this section, we use the definitions
\begin{align}
  L^2_N &:= \bigwedge_{i=1}^N L^2(S), \quad S = \R^3\times\Z_2 \, , \\
  H^k_N &:= H^k(S^N) \cap L^2_N\, .
\end{align}
Here, $H^k(S^N)$ denotes the set of elements of $L^2(S^N)$ whose $2^N$ spatial components are elements of the standard $k$'th order Sobolev space $H^k(\R^{3N})$, i.e., each spatial component have square-integrable weak partial derivatives up to order $k$~\cite{Schmuedgen2012}.
The starting point for DFT is the ground-state problem of $N$ interacting electrons in an external potential $v$ on variational form, i.e.,
\begin{equation}
    E[v] := \inf \left\{
        \mathcal{E}_0[\psi] + \mathcal{V}[\psi] \; \mid \; \psi \in \mathcal{W}_N
    \right\}\label{eq:Ev},
\end{equation}
with $\mathcal{W}_N = \{\psi \in H^1_N  \mid \|\psi\|_2 = 1\}$ consisting of the normalized finite kinetic-energy wavefunctions,
and with
\begin{align}
    \mathcal{E}_0[\psi] &= \frac{1}{2}\braket{\nabla\psi,\nabla\psi} + \braket{\psi,W\psi}, \\
    \mathcal{V}[\psi] &= \braket{\psi,V\psi} = \int \rho_\psi(\bfr)v(\bfr)\, d^3\bfr \,.
\end{align}
Here, $V$ is the many-electron operator corresponding to the potential $v$, and $\rho_\psi$ is the one-electron density of $\psi$.
In order to make connection with more standard formulations of quantum mechanics, the variational minimization in in terms of quadratic forms in Eq.~\eqref{eq:Ev} needs to be connected with the spectral theory of selfadjoint operators.
Intuitively, the kinetic energy part of the variational formulation is obtained by integration by parts, enlarging the domain of the kinetic energy operator term $T = -\tfrac{1}{2}\nabla^2$ in the full Hamiltonian. However, the transition back and forth between a selfadjoint operator and a quadratic form is subtle, and for some potentials the domain of the Hamiltonian \emph{operator} is not $H^2_N$ due to singular behavior of $v$. For this exposition, we let the following be sufficient: Whenever $v \in L^{3/2}(\R^3) + L^\infty(\R^3)$, $\mathcal{V}$ will be relatively bounded with respect to $\mathcal{E}_0$, which means that as perturbation it is sufficiently gentle to allow the sum of the forms to be well-defined with domain $H^1_N$ by the KLMN theorem~\cite{Schmuedgen2012,Simon1971}. The fundamental representation theorem of closed semibounded quadratic
forms~\cite{Schmuedgen2012} guarantees  the existence of a unique self-adjoint operator $\hat{H}[v] : D[v] \to L^2$ such that for every $\psi$ in the operator domain $D[v]$, we have $\mathcal{E}_0[\psi] + \mathcal{V}[\psi] = \braket{\psi,\hat{H}[v]\psi}$. We point out that $D[v]\subset H^1_N$ is dense in $L^2_N$, and that $v\in L^2(\R^3) + L^\infty(\R^3)$ is sufficient to guarantee $D[v] = H^2_N$, a fundamental result by Kato~\cite{Kato1951}. However, for the stronger singularities present in $L^{3/2}(\R^3)+ L^\infty(\R^3)$ it may happen that $D[v]$ becomes a proper subset of $H^2_N$. On the other hand, the \emph{form domain} is always $H^1_N$ for every $v \in L^{3/2}(\R^3) + L^\infty(\R^3)$. In any case, the infimum in Eq.~\eqref{eq:Ev} will be the
bottom, i.e., infimum, of the spectrum of $\hat{H}[v]$, and hence the connection is complete.

\subsection{Densities}

Central to DFT is the density $\rho_\psi \in L^1(\R^3)$ associated with a (normalized) $\psi \in L^2_N$.
\begin{definition}
For any $\psi \in L^2_N$, we define the density $\rho_\psi : \R^3 \to [0,+\infty)$ almost everywhere (a.e.) by the formula
\begin{equation}
  \rho_\psi(\bfr) := N \sum_{s\in\Z_2}\int_{S^{N-1}} |\psi((\bfr,s),\bfx_2,\cdots,\bfx_N)|^2\, d\bfx_2\cdots d\bfx_N\,.
\end{equation}
For a given $\rho \in L^1(\R^3)$, we write $\psi\mapsto \rho$ if $\rho_\psi = \rho$.
\end{definition}
It is immediate that $\rho_\psi \geq 0$ a.e., and that $\|\rho_\psi\|_1 = N\|\psi\|^2$.
In fact, we have a complete characterization of densities that come from elements $\psi \in \mathcal{W}_N$. In the proof of the following theorem, the proof that the density of $\psi \in \mathcal{W}_N$ satisfies $\rho_\psi^{1/2}\in H^1(\mathbb{R}^3)$ is due to G.~Friesecke.  
\begin{theorem}[$N$-representable densities]\label{thm:N-rep}
 Let $\mathcal{D}^N$ be the set of measurable functions $\rho : \R^3 \to \R$ that satisfy $\rho \geq 0$ a.e., $\rho^{1/2} \in H^1(\R^3)$ with $\|\rho\|_1 = N$. Then for every $\rho \in \mathcal{D}^N$ there is a $\psi \in \mathcal{W}_N$ such that $\psi \mapsto \rho $. Conversely, for every $\psi \in \mathcal{W}_N$, $\rho_\psi \in \mathcal{D}^N$. Moreover, $\mathcal{D}^N$ is convex.
\end{theorem}
\begin{proof}
For the construction of a $\psi\in\mathcal{W}_N$ with $\rho_\psi = \rho$ for a given $\rho\in\mathcal{D}^N$, see \cite[Theorem 1.2]{Lieb1983}. For the converse, we must prove that $\rho_\psi^{1/2} \in H^{1}(\R^3)$.

  Let $\psi \in \mathcal{W}_N$, and let $\rho = \rho_\psi$.
    Differentiating, we obtain
    \[ \nabla \rho(\bfr) = N \sum_{s\in \Z_2} \int_{S^{N-1}} (\psi^* \nabla_1 \psi + \psi \nabla_1 \psi^*) \, d\bfx_2\cdots d\bfx_N \,. \]
    Application of Cauchy--Schwarz gives
    \begin{equation}\label{eq:nabla-rho-bound} |\nabla \rho(\bfr)|^2 \leq 8 \rho(\bfr) t(\bfr), \end{equation}
    where
        \[ t(\bfr) =  \frac{N}{2}\sum_{s\in \Z_2} \int_{S^{N-1}} |\nabla_1\psi|^2 \, d\bfx_2\cdots d\bfx_N \in L^1(\R^3) \]
    is the kinetic-energy density. Consequently, $\rho \in W^{1,1}(\R^3)$, the standard Sobolev space of $L^1(\R^3)$ functions with first-order weak derivatives in $L^1(\R^3)$.
    
    Let $G(\bfr) = \exp(-|\bfr|^2)$, and consider $\rho_\epsilon(\bfr) = \rho(\bfr) + \epsilon G(\bfr)$, with $\epsilon>0$ a small parameter.
    We first prove that $\rho^{1/2}_\epsilon\to \rho^{1/2}$ in $L^2(\R^3)$ as $\epsilon\to 0$. It is clear that $|\rho^{1/2}_\epsilon - \rho^{1/2}| \to 0$ pointwise. By the elementary inequality $(a-b)^2 \leq 2(a^2 + b^2)$ we have, for all $\epsilon\leq 1$,
    \[ |\rho^{1/2}_\epsilon - \rho^{1/2}|^2 \leq 2(\rho + \epsilon G + \rho) \leq 2G + 4\rho \in L^1(\R^3). \]
    By the dominated convergence theorem, the integral of the left-hand side converges to zero, and hence $\rho^{1/2}_\epsilon\to \rho^{1/2}$ in $L^2(\R^3)$ as claimed.
    
    We next prove that $\rho_\epsilon^{1/2}$ is bounded in $H^1(\R^3)$ for $\epsilon \leq 1$. Let $h(z) = (z + \epsilon G(\bfr) )^{1/2}$, where $z\in[0,+\infty)$ and $\bfr$ is in some compact subset of $\R^3$. The function $h$ is continuously differentiable with bounded derivative. By the chain rule for Sobolev functions which says that (1) the composition of a $C^1$ function with bounded derivative and a function in $W^{1,1}_\text{loc}(\R^3)$ is again in $W^{1,1}_\text{loc}(\R^3)$, and (2) that its partial derivatives can be computed with the usual chain rule~\cite[Ch.~5, Exerciswe 17]{Evans1998}, we obtain $(\rho + \epsilon G)^{1/2}\in W^{1,1}(\R^3)_\text{loc}$ with $\nabla (\rho + \epsilon G)^{1/2} = \tfrac{1}{2}(\rho + \epsilon G)^{-1/2}(\nabla \rho + \epsilon \nabla G)$. Key now, is that this expression is uniformly bounded in $L^2(\R^3)$. Indeed,
    \begin{equation}\label{eq:rhoH1-estimate} |\nabla\rho_\epsilon^{1/2}|^2 \leq \frac{|\nabla\rho|^2 + |\epsilon \nabla G|^2}{2 (\rho + \epsilon G)} \leq \frac{\rho t}{N(\rho + \epsilon G)} + \epsilon \frac{|\nabla G|^2}{2G} \leq \frac{t}{N} + \frac{|\nabla G|^2}{2G} \in L^1(\R^3). \end{equation}
    Thus, $\rho_\epsilon^{1/2}$ is bounded in $H^1(\R^3)$ as claimed. By weak sequential compactness of bounded sets in the Hilbert space $H^1(\R^3)$ there exists a sequence $\{\epsilon_k\}\subset[0,1]$ such that $\rho_{\epsilon_k}^{1/2} \to u$ weakly in $H^1(\R^3)$. Since $\rho_{\epsilon_k}^{1/2} \to \rho^{1/2}$ strongly in $L^2(\R^3)$, $u = \rho^{1/2}$, and thus $\rho^{1/2} \in H^1(\R^3)$.

    For convexity of $\mathcal{D}^N$, we first note that $\rho \in W^{1,1}(\R^3)$, since $\rho^{1/2}\in\mathcal{D}^N$ implies that $\rho = \phi^2$ for a unique $\phi\in H^1$, $\phi \geq 0$ almost everywhere, and thus $\|\nabla \rho\|_1 = \|2\phi\nabla \phi\|_1\leq 2 \|\phi\|_2\|\nabla\phi\|_2<+\infty$. Let $\lambda \in [0,1]$, $\rho_1,\rho_2\in \mathcal{D}^N$, and set $\rho = \lambda \rho_1 + (1-\lambda)\rho_2$. We need to check that $\rho^{1/2} \in H^1(\R^3)$, which is obtained by repeating the regularization argument carried out for $\rho=\rho_\psi \in W^{1,1}(\R^3)$. For that, we need a bound on $|\nabla\rho(\bfr)|^2$ in lieu of Eq.~\eqref{eq:nabla-rho-bound}. We begin with the inequality
    \[ |\nabla \rho|^2 \leq 4 \lambda^2\phi_1^2|\nabla\phi_1|^2 +  4 (1-\lambda)^2 \phi_2^2 |\nabla\phi_2|^2 + 8 \lambda(1-\lambda)\phi_1\phi_2\nabla\phi_1\cdot\nabla \phi_2\,, \]
    and apply Young's inequality, i.e., $2\phi_1\phi_2 \nabla\phi_2\cdot\nabla\phi_2 \leq \phi_1^2|\nabla\phi_1|^2 + \phi_2^2|\nabla\phi_2|^2$, and then rearrange to get
    \[  |\nabla \rho|^2 \leq 4 \rho (\lambda |\nabla \phi_1|^2 + (1-\lambda)|\nabla\phi_2|^2)\,, \]
    which is precisely what is needed.
\end{proof}
\begin{remark}
  The need for some regularization to an everywhere positive density comes from the fact that the commonly used formula $\nabla\rho^{1/2} = \tfrac{1}{2}\rho^{-1/2}\nabla\rho$ only makes sense if $\{\bfr\in\R^3\mid \rho(\bfr)=0\}$ has zero measure.  When this can be verified, such as for $\rho=\rho_\psi$ for everywhere strictly positive, the proof of the above result is somewhat simpler~\cite[Theorem 1.1]{Lieb1983}. 
\end{remark}
Throughout this section, we let $X := L^1(\R^3)\cap L^3(\R^3)$, a Banach space with norm $\|\rho\|_X := \max(\|\rho\|_1,\|\rho\|_3)$. We note that the dual is $X^* = L^{3/2}(\R^3) + L^\infty(\R^3)$ with norm $\|v\|_{X^*} = \inf\{ \|f\|_\infty + \|g\|_{3/2} \mid f+g=v \}$, also a Banach space~\cite{Liu1968}.
\begin{proposition}
    $\mathcal{D}^N \subset X$.
\end{proposition}
\begin{proof}
    We need only show that $\rho \in L^3(\R^3)$. But $H^1(\R^3) \subset L^6(\R^3)$ by the Sobolev embedding theorem, i.e. $\rho \in \mathcal{D}^N$ implies $\|\rho^{1/2}\|_6 = \|\rho\|_3^{1/2} < +\infty$.
\end{proof}

\subsection{Constrained search and skew conjugate pairs}

Having a well-defined space of densities, we now introduce \emph{constrained search} into Eq.~\eqref{eq:Ev}, i.e.,
\begin{equation}
    E[v] = \inf_{\rho\in X} \left\{F_\text{LL}[\rho] + \int v \rho \right\},
    \label{eq:Ev-LL}
\end{equation}
where the Levy--Lieb functional is defined for any measurable $\rho$ by the expression
\begin{equation}
    \begin{split}
    F_\text{LL}[\rho] &:= \inf_{\psi\mapsto\rho} \left\{ \mathcal{E}_0[\psi] \right\} \\ & = \inf\left\{ \mathcal{E}_0[\psi] \; \mid \; \psi \in \mathcal{W}_N, \; \rho_\psi = \rho \; \text{a.e.} \right\}.
    \end{split}
\end{equation}
We use the convetion that $F_\text{LL}[\rho] = +\infty$ whenever $\rho \notin \mathcal{D}^N$.

We now have that $E : X^* \to \R$ is given as a pointwise infimum over a nonempty family of continuous affine functions. Such functions are automatically concave and upper semicontinuous.

For any pair $(v,\rho)\in X^*\times X$ we have the inequality $E[v] \leq F_\text{LL}[\rho] + \int v \rho$. Rearranging, we can define a new function $F : X \to \R \cup \{+\infty\}$ by the expression
\begin{subequations}
\begin{equation}
  F[\rho] := \sup_{v\in X^*} \left\{ E[v] - \int v \rho \right\}, \label{eq:E-to-F}
\end{equation}
which is then automatically \emph{convex} and \emph{lower} semicontinuous,  satisfying $F[\rho] \leq F_\text{LL}[\rho]$. As we will show, one also has
\begin{equation}
  E[v] = \inf_{\rho\in X} \left\{ F[\rho] + \int v \rho \right\}.
  \label{eq:F-to-E}
\end{equation}
\end{subequations}
For an extended-valued map such as $F$, we define the \emph{effective domain} $\dom(F)$ to be the points where $F$ is real valued.
\begin{definition}[Skew conjugation]
    Let $B$ be a Banach space with dual $B^*$, consisting of the continuous linear functionals on $B$. Denote by $\braket{\cdot,\cdot}$ the dual pairing. For any $f : B \to \R\cup\{+\infty\}$ not identically $+\infty$ ($f$ is then called proper) we define the (skew) concave conjugate functional $f^\wedge : B^* \to \R\cup\{-\infty\}$ by
    \begin{equation}
      f^\wedge[x^*] := \inf_{x\in B} \left\{ f[x] + \braket{x^*,x} \right\}.
    \end{equation}
    For any $g : B^* \to \R\cup\{-\infty\}$, we define the (skew) convex conjugate functional $g^\vee : B \to \R\cup\{+\infty\}$ by
    \begin{equation}
      g^\vee[x] := \sup_{x^*\in B^*} \left\{ g[x^*] - \braket{x^*,x} \right\}. 
    \end{equation}
    A pair of functionals $(f,g)$ satisfying $f = g^\vee$ and $g = f^\wedge$ are said to be a \emph{skew-conjugate pair of functionals}.
\end{definition}

  The above introduced concept of (skew) convex/concave conjugate functionals and its notation is slightly uconventional~\cite{Rockafellar1968}, but useful in DFT. The \emph{conventional} convex conjugate (Legendre--Fenchel transform) of $f : B \to \R\cup\{+\infty\}$ is the function $f^* : B^* \to \R\cup\{\pm\infty\}$ given by
  \begin{equation} f^*[x^*] := \sup_{x\in B} \left\{ \braket{x^*,x} - f[x] \right\}. \label{eq:conjugate} \end{equation}

We introduce the class $\Gamma_0(B)$ of proper convex lower semicontinuous functions $f : B \to \R\cup\{+\infty\}$, meaningful for the Banach space $B$ but also for general topological vector spaces such as $B_w^*$, the dual space $B^*$ equipped with the weak-$*$ topology. A central fact is that the Legendre--Fenchel transformation is a bijection between these two classes of functions. This result gives insight into the role of convex conjugation in DFT.
\begin{theorem}\label{thm:legendre-fenchel}
  The Legendre--Fenchel transform $f \mapsto f^*$ is a bijection between $\Gamma_0(B)$ and $\Gamma_0(B^*_w)$, and $(f^*)^* = f$, as well as $(g^*)^* = g$ for any $g \in \Gamma_0(B_w^*)$.
\end{theorem}
\begin{proof}
  We mention two facts. First, $f\in \Gamma_0(B)$ if and only if $f \in \Gamma_0(B_w)$, where $B_w$ is $B$ equipped with the weak topology. This follows from $f$ being lower semicontinuous if and only if $f$ has closed convex sublevel sets. By \cite[Proposition 1.73]{Barbu2012}, norm-closed convex sets are weakly closed and vice versa. 
  Second, a fact from functional analysis~\cite[Theorem IV.20]{ReedSimon} is that the dual of $B_w$ is $B_w^*$, and that $(B_w^*)^* = B_w$.
  
  From \cite[Corrolary 2.21]{Barbu2012}, if $V$ is a locally convex topological vector space, $f : V \to \R\cup\{+\infty\}$ is proper if and only if $f^* : V^* \to \R\cup\{+\infty\}$ is proper. Since $f^*$ is the supremum of a family of lower semicontinuous functions, $f\in\Gamma_0(B_w)$ implies $f^* \in \Gamma_0(B^*_w)$. By reflexivity the same argument gives that $g \in \Gamma_0(B_w^*)$ implies $g^* \in \Gamma_0(B_w)$. The biconjugation theorem (Fenchel--Moreau) for locally convex spaces~\cite[Theorem 2.22]{Barbu2012} can for reflexive locally convex topological spaces be phrased as $f \in \Gamma_0(V)$ implies $(f^*)^* = f$. Hence, the conjugation is a bijection between $\Gamma_0(B_w)$ and $\Gamma_0(B_w^*)$ and vice versa. Together with $\Gamma_0(B) = \Gamma_0(B_w)$, the bijection of $\Gamma_0(B)$ and $\Gamma_0(B_w^*)$ has been established.
\end{proof}
Theorem~\ref{thm:legendre-fenchel} can be reformulated in terms of skew conjugation.
In the following, we write $f \in -\Gamma_0(B)$ if $-f \in \Gamma_0(B)$, and so on.
\begin{proposition}\label{prop:legendre-fenchel2}
  Let $B$ be a Banach space.
  \begin{enumerate}
    \item\label{prop:legendre-fenchel2:1} Let $f \in \Gamma_0(B)$. Then $f^\wedge \in -\Gamma_0(B^*_w)$, and $f^\wedge[y] = -f^*[-y]$.
    \item\label{prop:legendre-fenchel2:2} Let $g \in -\Gamma_0(B^*_w)$. Then $g^\vee \in \Gamma_0(B)$ and $g^\vee[x] = (-g)^*[-x]$.
    \item\label{prop:legendre-fenchel2:3} The map $f \mapsto f^\wedge$ is a bijection between $\Gamma_0(B)$ and $-\Gamma_0(B^*_w)$, and the map $g \mapsto g^\vee$ is a bijection between $-\Gamma_0(B^*_w)$ and $\Gamma_0(B)$.
  \end{enumerate}
\end{proposition}
\begin{proof} Proof of \ref{prop:legendre-fenchel2:1}:
    \begin{equation*}
      \begin{split}
        f^\wedge[y] &= \inf_{x\in B} \left\{ f[x] + \braket{y,x} \right\} = \inf_{x \in B} \left\{ -(-f[x]) - \braket{-y,x} \right\} \\
        & = - \sup_{x \in B} \{ \braket{-y,x} - f[x] \} = - f^*[-y] \,.
      \end{split}
    \end{equation*}
      Proof of \ref{prop:legendre-fenchel2:2}:
    \begin{equation*}
      \begin{split}
        g^\vee[x] &= \sup_{y\in B^*} \left\{ g[y] - \braket{y,x} \right\} = \sup_{y \in B^*} \left\{ \braket{y,-x} - (-g[y]) \right\} \\
        & =  (-g)^*[-x]\,.
      \end{split}
    \end{equation*}
      Proof of \ref{prop:legendre-fenchel2:3}: Follows from \ref{prop:legendre-fenchel2:1}, \ref{prop:legendre-fenchel2:2}, and Theorem~\ref{thm:legendre-fenchel}.
\end{proof}

It is readily seen, that $F = E^\vee = (F^\wedge)^\vee = (F_\text{LL}^\wedge)^\vee$, and thus that $F = (F_\text{LL}^\wedge)^\vee$.  $F$ is distinct from $F_\text{LL}$, as the latter function is not convex. We have the following characterization:
\begin{theorem}\label{thm:envelope}
  For any $f : B \to \R\cup\{+\infty\}$, $f^{**} = (f^\wedge)^\vee \in \Gamma_0(B)$ is the largest convex lower semicontinuous minorant of $f$, often called the convex envelope or closed convex hull of $f$.
\end{theorem}
\begin{proof}
  See Theorem 6.15 in Ref.~\cite{VanTiel1984}.
\end{proof}
A central result in Lieb's analysis (with a proof attributed to Simon), is the following:
\begin{theorem}
    $F = F_\text{DM}$, the density-matrix constrained-search functional defined by
    \begin{equation}
        F_\text{DM}[\rho] := \inf_{\gamma\mapsto \rho} \left\{ \operatorname{Tr}(\hat{H}[0] \gamma) \right\},
    \end{equation}
    where the infimum extends over all trace-class operators $\mathcal{D}^N_\text{op}$ on the form $\gamma = \sum_k^\infty \lambda_k \ket{\psi_k}\bra{\psi_k}$, with $0\leq \lambda_k \leq 0$, $\sum_k \lambda_k = 1$, and $\{\psi_k\} \subset \mathcal{W}_N$ an $L^2$-orthonormal sequence. The notation $\gamma\mapsto\rho$ means that $\sum_k \lambda_k \rho_{\psi_k} = \rho$ almost everywhere. The effective domain of $F_\text{DM}$ is $\mathcal{D}^N$.
\end{theorem}
\begin{proof}
    See Theorem 4.4 in Ref.~\cite{Lieb1983}, in which it is proven that $F_\text{DM}$ is lower semicontinuous. Since it is also convex, we must have $F=F_\text{DM}$ by Theorem~\ref{thm:envelope}.
\end{proof}
The ground-state problem can be written as a minimization over the set $\mathcal{D}^N_\text{op}$ of density operators as
\begin{equation}
  E_\text{DM}[v] := \inf_{\gamma \in \mathcal{D}^N_\text{op}} \Tr(\hat{H}[v]\gamma) = \sum_k \lambda_k \left(\mathcal{E}_0[\psi_k] + \mathcal{V}[\psi_k]\right).
  \label{eq:E-DM}
\end{equation}
\begin{proposition}
  Equation~\eqref{eq:E-DM} and \eqref{eq:Ev} define the same function,  $E[v] = E_\text{DM}[v]$ for all $v\in X^*$.
\end{proposition}
\begin{proof}
  Clearly, $E_\text{DM}[v] \leq E[v]$, since $\mathcal{E}_0[\psi] + \mathcal{V}[\psi] = \Tr(\hat{H}[v]\ket{\psi}\bra{\psi})$, i.e., the search domain is larger in Eq.~\eqref{eq:E-DM}. On the other hand, $E_\text{DM}[v] \geq E[v]$, since for any $\gamma = \sum_k \lambda_k \ket{\psi_k}\bra{\psi_k}$, $\Tr(\hat{H}[v]\gamma) = \sum_k \lambda_k [\mathcal{E}_0[\psi_k] + \mathcal{V}[\psi_k]] \geq \sum_k \lambda_k E[v] = E[v]$.
\end{proof}
Both $F_\text{LL}$ and $F_\text{DM}$ have the important property that they are \emph{expectation valued}, i.e., that the infimums in their definition are attained~\cite{Kvaal2015} as expectation values of unique states:
\begin{theorem}\label{thm:expvalued}
  For every $\rho \in \mathcal{D}^N$, there exists a unique $\psi_\rho \in \mathcal{W}_N$ such that $F_\text{LL}[\rho] = \braket{\psi_0|\hat{H}[0]|\psi_0}$, and a unique $\gamma_\rho = \sum_k \lambda_k \ket{\psi_k} \bra{\psi_k}$ such that $F_\text{DM}[\rho] = \Tr(\gamma \hat{H}[0]) = \sum_{k} \lambda_k \braket{\psi_0,\hat{H}[0]\psi_0}$, with $\psi_k \in \mathcal{W}_N$.
\end{theorem}
\begin{proof}
  Theorem 3.3 and Corollary 4.5(ii) in Ref.~\cite{Lieb1983}.
\end{proof}

\subsection{Sub- and superdifferentiability}

Equations~\eqref{eq:E-to-F} and \eqref{eq:F-to-E} cannot in general be differentiated to find a critical point condition, even if this is routinely done in the physics literature, see for example the classic monograph~\cite{Parr1989} by Parr and Yang. Indeed, neither $F$ nor $E$ are differentiable in general. However,
for convex optimization problems, the weakest useful notion of differentiability is not the usual Gâteaux or Fréchet differentiability, but that of subdifferentiability.

\begin{definition}\label{def:subdiff}
  Let $B$ and $C$ be topological vector spaces.
  Let $f: B \to \R\cup\{+\infty\}$, and let $x \in B$. The \emph{subdifferential} of $f$ at $x$ is the set
  \begin{equation}
    {\subdiff} f[x] = \left\{ y \in B^* \mid \forall x' \in B, \, f[x] + \braket{y,x'-x} \leq f[x']   \right\}.
  \end{equation}
  The elements are called \emph{subgradients}.
  Similarly, the \emph{superdifferential} of $g : C \to \R\cup\{-\infty\}$ at $x$ is the set
  \begin{equation}
    {\supdiff} g[x] = \left\{ y \in C^* \mid \forall x' \in C, \, g[x] + \braket{y,x'-x} \geq g[x']   \right\}.
  \end{equation}
  The elements are called \emph{supergradients}.
\end{definition}
In other words, the subdifferential is the set of slopes of  tangent functionals of $f$ at $x$ that are nowhere above the graph of $f$, i.e., below-supporting tangent functionals. Similarly, the superdifferential consists of the set of slopes of above-supporting tangent functionals. Note, that for $f \in \Gamma_0(B^*_w)$, $C = B^*$ with weak-$*$ topology, and thus ${\supdiff} f[x] \subset B$.

Important properties of the subdifferential are be summarized in the following proposition, whose proof is so easy we skip it:
\begin{proposition}\label{prop:subdiff}
  Suppose $f \in \Gamma_0(B)$, $g = f^\wedge \in -\Gamma_0(B_w^*)$ (such that $f = g^\vee$) form a skew-conjugate pair of functionals. Then the following holds:
  \begin{enumerate}
    \item\label{prop:subdiff:1}
      The subdifferential ${\subdiff} f : B \to 2^{B^*}$ is a monotone operator, i.e., for all $x_1,x_2\in B$, and for all $y_i \in {\subdiff} f[x_i] \subset B^*$,
      \begin{equation}
         \braket{y_1-y_2,x_1-x_2} \geq 0\,. \label{eq:monotone}
      \end{equation}
      The superdifferential ${\supdiff} g : C \to 2^{C^*}$ is similarly monotone, i.e., for all $x_1,x_2\in C$, and for all $y_i \in {\supdiff} g[x_i]\subset C^*$,
      \begin{equation}
         \braket{y_1-y_2,x_1-x_2} \leq 0\,. \label{eq:monotone2}
      \end{equation}
      If $f$ ($g$) is additionally strictly convex (concave), then Eq.~\eqref{eq:monotone} [\eqref{eq:monotone2}] holds strictly if $x_1\neq x_2$.
    \item
      Let $x\in B$ and $y \in B^*$ be given.
      Then Fenchel's inequality holds,
      \[ g[x] - f[y] \leq \braket{y,x} \]
    \item\label{prop:subdiff:3}
 The following are equivalent:
    \begin{enumerate}
      \item\label{prop:subdiff:3:1} $g[y] - f[x] = \braket{y,x}$
    \item\label{prop:subdiff:3:2} $-y \in {\subdiff} f[x]$
    \item\label{prop:subdiff:3:3} $x \in {\supdiff} g[y]$
    \item $f[x] + \braket{y,x}\leq f[x'] + \braket{y,x'}$ for all $x' \in B$
    \item $g[y] - \braket{y,x}\geq g[y'] - \braket{y',x}$ for all $y' \in B^*$
    \end{enumerate}
  \end{enumerate}
\end{proposition}
\begin{proof} Easy exercise. \end{proof}

The equivalence of \ref{prop:subdiff:3}\ref{prop:subdiff:3:2} and \ref{prop:subdiff:3}\ref{prop:subdiff:3:3} is particularly important, showing that two optimization problems are equivalent. Applied to exact DFT, we obtain the equivalence
\[ -v \in {\subdiff} F[\rho] \iff \rho \in {\supdiff} E[v] \iff E[v] = F[\rho] + \int v\rho\,, \]
giving the critical point conditions of Eqs.~\eqref{eq:E-to-F} and \eqref{eq:F-to-E}. We also have the following result, which relates the optimality condition to ground states on density operator form~\cite{Kvaal2015}:
\begin{proposition}\label{prop:ground-states}
    Let $\rho \in \mathcal{D}^N$ and $v \in X^*$ be given. Define $G[v] \subset\mathcal{D}^N_\text{op}$ as the set of minimizers for Eq.~\eqref{eq:E-DM}.
    \begin{enumerate}
    \item\label{prop:ground-states:1}
          $G[v] = \operatorname{co}\{ \ket{\psi}\bra{\psi} \mid \psi \in \mathcal{W}_N, \, E[v] = \braket{\psi,\hat{H}[v]\psi} \}$, the convex hull of pure-state ground-state density operators.
    \item\label{prop:ground-states:2}
      $E[v] = F[\rho] + \int v\rho$ if and only if there is a $\gamma \in G[v]$, with $\gamma \mapsto \rho$.
\end{enumerate}
\end{proposition}
\begin{proof}
  Proof of \ref{prop:ground-states:1}:
    Clearly the stated convex hull is a subset of $G[v]$, since for $\gamma = \sum_k \lambda_k \ket{\psi_k}\bra{\psi_k}$ with all the $\psi_k$ ground-states, $\Tr(\hat{H}[v]\gamma) = E[v]$. On the other hand, assume that $\gamma$ is not in this convex hull. Then for at least one $\psi_k$ in its decomposition, $\braket{\psi_k,\hat{H}[v]\psi_k} > E[v]$, and $\Tr(\hat{H}[v]\gamma) > E[v]$.
  Proof of \ref{prop:ground-states:2}:
  If: If $G[v]$ is empty, then $E[v] < F[\rho_\gamma] + \int\rho_\gamma v$ for any $\gamma$. Let $\gamma \in G[v]$, $\gamma\mapsto\rho$. Then $E[v] = \Tr(\hat{H}[v]\gamma) = \Tr(\hat{H}[0]\gamma) + \int \rho v$. Now, by definition $F[\rho] \leq \Tr(\hat{H}[0]\gamma)$, but clearly equality must hold, otherwise we obtain a contradiction. Thus $E[v] = F[\rho] + \int \rho v$.
  Only if: Since $F$ is expectation valued (Theorem~\ref{thm:expvalued}), there is a $\gamma \mapsto \rho$ such that $F[\rho] = \Tr(\hat{H}[0]\gamma)$, and hence $E[v] = \Tr(\hat{H}[0]\gamma) + \int \rho v = \Tr(\hat{H}[v]\gamma)$. Thus, $\gamma \in G[v]$.
\end{proof}

\subsection{Regularity of $E$ and $F$}

We next state some regularity results for convex functions in $\Gamma_0(B)$, with $B$ being a Banach space throughout this section.
\begin{definition}[Gâteaux and Fréchet differentiability]\label{def:gateaux}
 A proper  functional $f : B \to \R \cup \{+\infty\}$ is called \emph{Gâteaux differentiable}
 at a point $x\in\dom(f)$, if for all $h$ in $B$, the directional derivative $f'(x;h) = \lim_{t\to 0^+} t^{-1}[f(x+th)-f(x)]$ exists, such that $f'(x;h) = \braket{\nabla f(x),h}$ for some $\nabla f(x) \in B^*$ called the \emph{Gâteaux derivative}. If additionally
 \[ f(x+h) = f(x) + \braket{\nabla f(x),h} + o(\|h\|) \]
 as $\|h\|\to 0$,
 then $f$ is said to be \emph{Fréchet differentiable} at $x$.\end{definition}
Note that we require the directional derivative to be a continuous linear functional. Some authors do not require the directional derivative to be continuous or even not continuous in the definition of Gâteaux differentiability.
\begin{theorem}\label{thm:singleton}
    Let $f\in \Gamma_0(B)$, and $x \in B$. If $f$ is G\^ateaux differentiable at $x$ then ${\subdiff} f(x) = \{ g\}$, a singleton, where $g = \nabla f (x_0)$, the G\^ateaux derivative. Conversely, if $f$ is continuous at $x_0$ and ${\subdiff} f(x_0)$ is a singleton, then $f$ is G\^ateaux differentiable at $x_0$, and $\{\nabla f(x_0)\} = {\subdiff} f(x_0)$.
\end{theorem}
\begin{proof}
  See Proposition 2.40 in Ref.~\cite{Barbu2012}.
\end{proof}
\begin{remark}
  Note the continuity requirement in the converse statement. The subgradient being a singleton is not alone sufficient to guarantee G\^ateaux differentiability.
\end{remark}

The following theorem demonstrates that convexity together with local boundedness above is a quite strong assumption on an $f\in \Gamma_0(B)$.
\begin{theorem}\label{thm:finite-cont}
  Suppose $f \in \Gamma_0(B)$, with $B$ a Banach space.
  \begin{enumerate}
    \item\label{thm:finite-cont:1} If $f$ is locally bounded above near $x\in B$,  then $f$ is locally Lipschitz continuous near $x$, and  ${\subdiff} f(x)$ is nonempty.
    \item\label{thm:finite-cont:2} If $f$ is defined on a convex open set $C\subset B$, and is locally bounded above near \emph{some} $x\in C$, then $f$ is locally Lipschitz near \emph{all} of $x \in C$.
  \end{enumerate}
\end{theorem}
\begin{proof}
  Proof of \ref{thm:finite-cont:1}:
  Let $x\in B$ be given, and let $B$ be the closed ball of radius $\delta>0$ around $x$. Suppose $f(x)\leq M$ in $B$. Let $x + h \in B$ be arbitrary.
  Since $ x = (x + h)/2 + (x - h)/2$, convexity of $f$ gives
  $f(x) \leq f(x+h)/2 + f(x - h)/2$. Thus
  \[ f(x + h) \geq 2 f(x) - f(x-h) \geq 2 f(x) - M, \]
  which implies
  \[ |f(x + h)| \leq \max\{ M, M - 2 f(x) \} \leq |M| + 2|f(x)| =: M'. \]
  Let $B'\subset B$ be a slightly smaller concentric ball of radius $\delta - \epsilon$. Let $y_1,y_2\in B'$. Consider the point
  \[ z = y_1 + \frac{\epsilon }{\|y_1 - y_2\|}(y_1 - y_2)\,. \]
  Now $\|z - x\| \leq \delta$ so that $x \in B$, and $y_2$ lies in the open interval $(y_1,z)$. Explicitly, $y_2$ is given by the convex combination
  \[ y_2 = \frac{\epsilon}{\epsilon + \|y_2-y_1\|} y_1 + \frac{\|y_2-y_1\|}{\epsilon + \|y_2-y_1\|} z\,, \]
  Convexity of $f$ now gives
  \[ f(y_2) \leq \frac{\epsilon}{\epsilon + \|y_2-y_1\|} f(y_1) + \frac{\|y_2-y_1\|}{\epsilon + \|y_2-y_1\|} f(z)\,, \]
  which after rearrangement gives
  \[ f(y_2)-f(y_1) \leq (f(z) - f(y_2))\|y_2-y_1\| \leq \frac{2 M'}{\epsilon} \|y_2-y_1\|\,. \]
  Repeating the argument with $y_1$ and $y_2$ interchanged gives the desired Lipschitz continuity.

  The existence of subgradients near $x$ is a consequence of the geometric form of the Hahn--Banach theorem, see, e.g., Theorem 1.36 in Ref.~\cite{Barbu2012}. For the complete proof, see Proposition 2.36 in Ref~\cite{Barbu2012}.

  Proof of \ref{thm:finite-cont:2}: Let $B$ be a ball of radius $\delta$ around $x\in C$ on which $f$ is locally bounded above by a constant $M$. Let $y \in C$. There exists $z \in C$ such that $y = \lambda z + (1-\lambda)x$ for some $\lambda \in [0,1]$. Now, it is straightforward to see, that for all $x'\in C$ such that $\|x' - y\| \leq (1-\lambda)\delta$, $f(x') \leq \max\{ M, f(z) \}$.
  Thus, $f$ is locally bounded above near $y$, and by (1), locally Lipschitz near $y$. Since $y$ was arbitrary, we are done.
\end{proof}

We now discuss the behavior of the functionals $E \in -\Gamma_0(X^*_w)$ and $F \in \Gamma_0(X)$ terms of classical differentiability and subdifferentials. Note that $E$ is upper semicontinuous in the weak-$*$ topology on $X^*$, and indeed by \cite[Theorem 2.16]{Barbu2012} it is continuous in the stronger norm topology on $X^*$ since $E$ is everywhere defined and cleary upper semicontinuous as the infimum of a nonempty family of affine functions. However, we can say even more:
\begin{theorem}\label{thm:E-reg}\leavevmode
  \begin{enumerate}
    \item\label{thm:E-reg:1} The map $E : X^* \to \R$ is locally Lipschitz continuous.
    \item\label{thm:E-reg:2} Suppose $v \in X^*$.   Then $E$ is G\^ateaux differentiable at $v$ if and only if $\hat{H}[v]$ has a smallest eigenvalue, and all normalized eigenvectors belonging to this eigenvalue share the same density.
  \end{enumerate}
\end{theorem}
\begin{proof}
  Proof of \ref{thm:E-reg:1}: $E$ is everywhere finite. In order to apply Theorem~\ref{thm:finite-cont}, we need to show that $E$ is locally bounded above at \emph{some} point in $X^*$, and we choose $v = 0$. The proof is adapted from Ref.~\cite{Lieb1983}.

  Let $\|v\|_{X^*} < L/3$, where $L$ is the constant in the Sobolev embedding of $H^1(\R^3)$ in $L^6(\R^3)$ ($\|\nabla f\|_2^2 \geq L \|f\|_6^2$). Write $v = u + w$ with $u\in L^{3/2}$ and $w \in L^\infty$, and note that $\|u\|_{3/2} + \|w\|_\infty < L/3$. Using the fact that the two-electron repulsion operator is positive, we get
  \begin{equation}
    \begin{split}
      E[v] &= \inf_{\psi\in \mathcal{W}_N} \frac{1}{2}\|\nabla\psi\|_2^2 + \braket{\psi,W\psi} + \int \rho_\psi v \\
          &\geq \inf_{\psi\in\mathcal{W}_N}  \frac{1}{2}\|\nabla\psi\|_2^2 + \int \rho_\psi v \\
          &\geq \inf_{\psi\in\mathcal{W}_N} \frac{1}{2}\|\nabla\psi\|_2^2 - N \|w\|_\infty - \|\rho_\psi^{1/2}\|^2_6\|u\|_{3/2} \\
          &\geq \inf_{\psi\in\mathcal{W}_N} \frac{1}{2}\|\nabla\psi\|_2^2 - N \|w\|_\infty - L^{-1}\|\nabla\rho^{1/2}\|^2_2\|u\|_{3/2} \,.
    \end{split}
  \end{equation}
  Using Eq.~\eqref{eq:rhoH1-estimate}, we get
  \begin{equation}
    E[v] \geq \inf_{\psi\in\mathcal{W}} \frac{1}{2}\|\nabla\psi\|^2_2 - \frac{1}{3}NL - \frac{1}{3}\Big(\frac{1}{2N}\|\nabla\psi\|_2^2 + g\Big) \geq -\frac{1}{3}NL-\frac{1}{3}g\,,
  \end{equation}
  where $g = \int G^{-1}|\nabla G|<+\infty$, $G(\bfr) = \exp(-|\bfr|^2)$.
  Hence, $E$ is locally bounded above at $0\in X^*$.

  Proof of \ref{thm:E-reg:2}: 
  Since $E$ is continuous, Theorem~\ref{thm:singleton} tells us that $E$ is G\^ateaux differentiable at $v \in X^*$ if and only if the superdifferential is a singleton. It follows from Proposition~\ref{prop:ground-states} that ${\supdiff} E[v]$ is the convex hull of all densities $\rho_\psi$ of all ground-state wavefunctions of $\hat{H}[v]$. Thus, the superdifferential is a singleton if and only if all ground-state densities are the same.
\end{proof}

The universal functional $F$ is quite badly behaved. Our discussion is mostly based on Ref.~\cite{Lammert2007}, in which many more details can be found. Since $F$ is only defined on elements $\rho \in X$ for which $\int\rho = N$, it is clear that $F$ cannot be G\^ateaux differentiable, since any change in the particle number $\int \rho$ would give infinities. On the other hand, it could be differentiable in a more restricted sense, e.g., we could consider $F$ as a function on the mentioned affine space and study directional derivatives, or even on the smaller space $\operatorname{aff}\operatorname{dom} F$ (the affine hull of the effective domain). The following discussion shows that $F$ is singular in these cases, too.
\begin{theorem}\label{thm:Fdom}
  Let $X_N = \{\rho \in X \mid \int \rho = N \}$, an affine closed space of codimension 1. Let $X_N^+ \subset X_N$ be the subset of those elements that satisfy $\rho\geq 0$ a.e.
  \begin{enumerate}
    \item\label{thm:Fdom:1} The map $F : X \to \R\cup\{+\infty\}$ has effective domain $\mathcal{D}^N$, i.e., $F[\rho] < +\infty$ if and only if $\rho\in\mathcal{D}^N$.
    \item\label{thm:Fdom:2} $\mathcal{D}^N$ is dense in $X_N^+$.
    \item\label{thm:Fdom:3} $\operatorname{aff} \mathcal{D}^N$ has empty algebraic interior in $\operatorname{aff} X_N^+ = X_N$. 
    The algebraic interior of a subset $A \subset X_+$ consists of the points $\rho \in A$ such that, for
every line $\ell \subset X_N$ through $\rho$, $\ell \cap A$ contains a line segment with $\rho$ in its interior.  \end{enumerate}
\end{theorem}

\begin{proof}
  Proof of \ref{thm:Fdom:1}: The domain of $F_\text{DM}$ is the convex hull of the domain of $F_\text{LL}$, which is $\mathcal{D}^N$, a convex set.

  Proof of \ref{thm:Fdom:2}: For any $\sigma\in X_N^+$, we must find $\rho_\epsilon \in \mathcal{D}^N$ such that $\|\rho_\epsilon -\sigma\|\to 0$ as $\epsilon\to 0$. This can be done using standard mollification arguments: Let $g_\epsilon \in \mathcal{C}^\infty_c(\R^3)$ be a mollifier, and define $\phi = \rho^{1/2}/N^{1/2}$, $\psi_\epsilon = g_\epsilon * \phi$, $\phi_\epsilon = \psi_\epsilon/\|\psi_\epsilon\|$, $\rho_\epsilon = N \phi_\epsilon^2$.
  Now, $\phi_\epsilon \in H^1$, $\|\phi_\epsilon\|^2_2 = 1$, and moreover $\phi_\epsilon \geq 0$ a.e., so $\rho_\epsilon \geq 0$ a.e., and hence $\rho_\epsilon \in \mathcal{D}^N$. Furthermore, $\phi_\epsilon \to \phi$ in $H^1$, and $\|\rho_\epsilon-\rho\|_3 \leq C \|\phi_\epsilon - \phi\|^2_{H^1}$ by the same Sobolev embedding used in Theorem~\ref{thm:N-rep}. Similarly, $\|\rho_\epsilon - \rho\|_1 \to 0$.

  Proof of \ref{thm:Fdom:3}, adapted from Ref.~\cite{Lammert2007}: Let $\rho_0 \in \mathcal{D}^N$ be given. We need to find a line segment $\{\rho_0 + s\delta\rho\mid s \in [0,1]\} \in \operatorname{aff} \mathcal{D}^N$ such that $\rho_0 + s\delta\rho \notin \mathcal{D}^N$ for any $s > 0$.

  Let $\sigma \in \mathcal{C}^\infty(\R^3)$ be such that $\sigma(\bfr) \in [0,1]$, $\sigma(\bfr)=0$ outside the ball around the origin with unit volume, $\sigma(\bfr) \geq 1/2$ inside the ball around the origin with volume $1/2$ (implying that $\int\sigma \geq 1/4$). Furthermore, we require that $\sigma^{1/2} \in H^1$. Find a sequence $B_n$ of balls of volume $\mu(B_n) = 2^{-n}$ with centers $\bfr_n$ satisfying $|\bfr_n - \bfr_m| \geq 1$ for all $n,m$. Additionally we require that the measure of $\{\rho_0 \geq 2^{-3n}\}\cap B_n \leq \mu(B_n)/4 = 2^{-n-2}$ (by taking a subsequence if necessary). Such a sequence exists, since otherwise $\rho_0$ cannot have finite integral.

  We now define
  \[ \delta\rho(\bfr) = \tau(\bfr) - \sum_n 2^{-2n} \sigma(2^n(\bfr- \bfr_n))\,, \]
  where $\tau$ is smooth, nonnegative and with support that does not intersect any of the $B_n$, such that $\int\delta\rho = 0$. Moreover, require that $\tau^{1/2} \in H^1$. Now, $\delta\rho\in\operatorname{aff}\mathcal{D}^N$.

  On any given $B_n$, there is a region (in the inner part) of positive measure where
  \[ \rho(\bfr) \leq 2^{-3n} - s 2^{-2n - 1}. \]
  Clearly, for any given $s >0$ we can take $n$ sufficiently large so that these values are negative. Thus, $\rho \notin \mathcal{D}^N$ for any $s>0$.
\end{proof}

\begin{remark}
  The construction of the direction $\delta\rho$ that immediately exits $\mathcal{D}^N$ exploits the requirement that $\rho\geq 0$ almost everywhere, and thus a pathology of the \emph{domain} as opposed to the behavior of $F$ \emph{on the domain}. It is a fact that  the restriction of $F$ to $\mathcal{D}^N$ (with the $X$-topology) is everywhere discontinuous in the sense that we can construct a sequence $\rho_n$ in $\mathcal{D}^N$ which is $X$-convergent to some $\rho \in X_N^+$, but for which $F[\rho_n]\to+\infty$. Such a construction is outlined in Ref.~\cite{Lammert2007}.  This also indicates that the topology of $X$ is not really that well suited for DFT, since unlike $F$, the chosen topology on $X$ is insensitive to density gradients.
\end{remark}

\subsection{The conjecture of Hohenberg and Kohn}

The last discussion in this section pertains to the structure of ${\subdiff} F[\rho]$, and this is the domain of the Hohenberg--Kohn theorem~\cite[Theorem 3.2]{Hohenberg1964,Lieb1983}: That the external potential $v$ is a unique function of the density $\rho$, up to a constant shift. This is a very appealing notion, because if $\rho$ determines $v$, then it determines also $\hat{H}[v]$ and hence the ground-state wavefunction $\psi$, and hence all physical observables as functions $\Omega[\rho]$. The system density is elevated to a basic variable, replacing the wavefunction as the state parameter of the quantum mechanics of $N$-electron systems. Unfortunately, the Hohenberg--Kohn theorem as originally stated is a \emph{conjecture}, since its proof has one step which is not rigorous: One divides the Schr\"odinger equation by $\psi$ pointwise, but this requires the unique continuation property, see \cite[Theorem 3.2]{Lieb1983}. At the time of writing it is still a partially open question if the Hohenberg--Kohn conjecture is true for potentials $v \in X^*$, although Garrigue has established the unique continuation property for potentials $v \in L^p_\text{loc}(\R^3)$ with $p>2$, which include the Coulomb potentials. See also the discussion by Lammert~\cite{Lammert2018}.

\begin{conjecture}[Hohenberg--Kohn]\label{conj:hk}
  The density $\rho \in \mathcal{D}^N$ determines the potential $v\in X^*$ up to an addive constant, i.e., we have either
  \[ {\subdiff} F[\rho] = \{ -v + \mu \mid \mu \in \R \}\quad \text{for some $v\in X^*$}\,, \]
  or
  \[ {\subdiff} F[\rho] = \emptyset\,, \]
  whenever there are no potentials for which $\rho\in \mathcal{D}^N$ is a ground-state density.
\end{conjecture}
The concept of $v$-representability receives a lot of attention in the DFT literature.
\begin{definition}\label{def:v-rep}[Ensemble $v$-representability]
  We say that $\rho \in X$ is \emph{(ensemble) $v$-representable} if there is $v \in X^*$ such that $\rho=\rho_\gamma$ for a $\gamma \in G[v]$, i.e., that $\gamma$ is a ground-state density operator of $v$. Equivalently,
  \[ E[v] = F[\rho] + \int v \rho\,, \]
  or, $-v\in {\subdiff} F[\rho]$, or $\rho \in {\supdiff} E[v]$. The set of (ensemble) $v$-representable densities is thus $\mathcal{B}_N := \dom({\subdiff} F)$ (defined as those $\rho\in X$ for which ${\subdiff} F[\rho]\neq\emptyset$).
\end{definition}
\begin{remark}
  There herein introduced concept of (ensemble) $v$-representability is slightly less general than the one used in cassical DFT literature. Here, $\rho \in L^1$ is (ensemble) $v$-representable if it is the ground-state density of ``some potential'' $v : \R^3\to \R$, with the function space otherwise unspecified. For example, a gaussian density is the ground-state density of a harmonic potential (for $N=1$), but this potential is not in $X^*$. Lammert~\cite{Lammert2010} distinguishes between (ensemble) $v$-representability and (ensemble) $X^*$-representability, which is identical to our concept.
\end{remark}

One of the classical problems of DFT is to characterize the set $\mathcal{B}_N$, as this is the effective domain of the classical Hohenberg--Kohn functional $F_{\text{HK}}$. The motivation for this study has been that since one wishes to differentiate $F_\text{HK}$, the set $\mathcal{B}_N$ needs to be sufficiently well-behaved. In particular, it needs to have an algebraic interior for directional derivatives to make sense. Since we have $F \leq F_\text{HK}$ and in particular $\dom(F_\text{HK})\subset \dom(F)$, the algebraic interior of $\mathcal{B}_N$ must be empty. However, we can say the following based on the Brøndsted--Rockafellar theorem~\cite[Theorem 16.45]{Bauschke2011}, which implies that for $f\in\Gamma_0(B)$ (with $B$ a Banach space), $\dom(\subdiff f)$ is dense in $\dom(f)$.
\begin{proposition}
  $\mathcal{B}_N$ is dense in $\mathcal{D}^N$.
\end{proposition}
\begin{proof}
  Simply note that $\mathcal{B}_N = \dom({\subdiff} F)$ and $\mathcal{D}^N = \dom(F)$ .
\end{proof}

The Hohenberg--Kohn conjecture is often taken to be the foundation of DFT in the sense that it implies the existence of a universal Hohenberg--Kohn functional $F_{\text{HK}}$, but the present author considers this a red herring: Our developments so far does not rely on it, and there is actually not much extra to gain from elevating the conjecture to a theorem, since the mapping $\rho \mapsto v$ would be very ill-behaved, see the discussion by Lammert in Ref.~\cite{Lammert2007}. First, not all $v\in X^*$ have ground states. Second, not all $\rho \in \mathcal{D}^N$ can be ground-state densities (``$v$-representable''), and we have little knowledge of how the one-dimensional affine space ${\subdiff} F[\rho]$ changes with $\rho$.

 %
 %
 %
 %
 %
 %
 %
 %

\subsection{Some remarks on exact DFT}

The previous section sets up exact DFT as a convex optimization problem, demonstrating a symmetry between the ground-state energy map $E[v]$ and the ``universal'' functional $F[\rho]$. The framework differs significantly from the ``traditional'' treatment that was initiated by the publications of Hohenberg, Kohn, and Sham~\cite{Hohenberg1964,Kohn1965}.

The framework has both strengths and weaknesses from the point of view formulating the quantum physics of $N$-electron systems in external fields. First, it is a mathematically rigorous formulation of exact DFT that pinpoints some of the deficiencies of the ``traditional'' formulation of DFT, in terms of a well-established body of results.

Second, the potential space $X^*$ feels both unnecessary large and too small at the same time, containing not only the physical Coulomb potentials, but also a host of ``wild'' potentials that would never appear in actual physical problems, while lacking some obvious interesting potentials, like the harmonic oscillator. At the same time, the constrained-search formula~\eqref{eq:Ev-LL} is valid for the latter, which pose no particular difficulties for quantum theory. Actually, Eq.~\eqref{eq:Ev-LL} is valid for much more general potentials. For example, for $v = v_+ - v_- \in X^*$ with $v_\pm \in X^*$ almost everywhere positive, one may modify $v_+$ or $v_0$ (but not both) to arbitrary elements of $L^1_\text{loc}(\R^3)$ (even if we may not have a link to a self-adjoint $\hat{H}[v]$ in such cases). It seems fortuitous that $X^*$ happens to contain the most interesting potentials, namely those of the Coulomb potentials, as $X$ is chosen for its capacity to hold $\mathcal{D}^N$.

Third, the usual norm topology on $X$ is not so suitable for DFT. It seems difficult to introduce a ``well-behaved'' Banach space of densities with a dual that contains correspondingly ``mostly'' interesting potentials. Thus, the standard convex analytic setting could be refined to allow for a better description of exact DFT.

Finally, it seems hard to connect the above treatment of exact DFT with the development of density-functional approximations. Indeed, it is not clear the existing functional approximations are at all approximations to $F_\text{DM}$ (and one may suspect that they are not).

\hilite{MOVE AND/OR EXPAND?}

At this point, it is worthwhile to mention alternative formulations of DFT, in particular the coarse-grained formulation of Lammert~\cite{Lammert2006,Lammert2010}. Coarse-graining can be motivated from (a) that experimental resolution is not infinite, and (b) that the nucleus is not a point particle, so that the singular Coulomb potential is not even exact. The space $\R^3$ is covered by a disjoint set of cells of uniformly bounded volume, such as a rectangular grid of uniform cells, and one considers equivalence classes $\boldsymbol{\rho}$ of those $\rho \in \mathcal{D}^N$ that have the same average in each cell. The potentials are taken to be constant over each cell. In this setting, every strictly positive $\boldsymbol{\rho}$ is shown to be ensemble $v$-representable. Moreover, the universal functional becomes much more well-behaved, in particular it is G\^ateaux differentiable. Furthermore, Lammert demonstrates that certain limits conforming with Lieb's theory are obtained as the resolution is increased. In summary, Lammert's coarse-grained DFT represents a regularization alternative to Moreau--Yosida regularization which we study Section~\ref{sec:my-dft}

 %
 %
 %
 %
 %
 %
 %

\subsection{Box truncated exact DFT}
\label{sec:box}

One basic problem with the spaces $X$ and $X^*$ is that they are nonreflexive, and Moreau--Yosida regularization has the most powerful effect in a Hilbert space setting. We therefore consider the $N$-electron problem in a open, connected and bounded domain $\Omega \subset \R^3$, assuming Dirichlet boundary conditions on the Schr\"odinger equation. All the results from exact DFT carry over to this situation. On the other hand, the finite domain allows some simplifications and stronger statements. In particular, the the long-range part of Coulomb potentials, being the source of the nonreflexiveness, disappear. In fact, we may now include harmonic potentials and other unbounded potentials into the class of external potentials. Moreover, every Hamiltonian $\hat{H}[v]$ will have a ground state.

We begin with a simple lemma.
\begin{lemma}\label{lemma:l2l1} Let $X(\Omega) := L^1(\Omega)\cap L^3(\Omega)$.
  Then $X(\Omega) \subset L^2(\Omega)$, for $\Omega$ bounded or unbounded, and also and $L^2(\Omega) \subset X(\Omega)'$. For $\Omega$ bounded, we have $L^2(\Omega) \subset L^1(\Omega)$, with continuous embedding. Furthermore, $L^2(\Omega) \subset X(\Omega)^*$ with continuous embedding.
\end{lemma}
\begin{proof}
  Let $u \in X(\Omega)$, i.e., $\|u\|_1$ and $\|u\|_3$ are both finite.
  \[ \|u\|_2^2 = \int_{|u|\leq 1} |u|^2 + \int_{|u|>1} |u|^2 \leq \int_{|u|\leq 1} |u| + \int_{|u| > 1} |u|^3 < +\infty\,. \]
  Thus $u \in L^2(\Omega)$.
  It is a standard fact that $L^q(\Omega) \subset L^p(\Omega)$ for $1 \leq p < q < +\infty$ with continuous embedding for bounded $\Omega$, proven by a simple application of Hölder's inequality. In particular, $L^2(\Omega)\subset L^1(\Omega)$. Morevover, $L^2(\Omega) \subset L^{3/2}(\Omega) \subset X(\Omega)^*$, and since $\|u\|_{3/2} \leq \|u\|_{X^*(\Omega)}$, the embedding is continuous.
\end{proof}
The significance of Lemma~\ref{lemma:l2l1} is that it makes sense to consider $F_\text{DM}$ as a function of $\rho \in L^2(\Omega)$, since $\mathcal{D}^N(\Omega) \subset L^2(\Omega)$ (with an obvious definition of $\mathcal{D}^N(\Omega)$):
\begin{proposition}
  The functional $F_\text{DM} : L^2(\Omega) \to \R\cup\{+\infty\}$ is lower semicontinuous.
\end{proposition}
\begin{proof}
  Suppose $\rho_n\to\rho$ in $L^2(\Omega)$. Then $\rho_n\to \rho$ in $L^1(\Omega)$ by Lemma~\ref{lemma:l2l1}, and $F[\rho]\leq \liminf_n F_\text{DM}[\rho_n]$ by lower semicontinuity of $F_\text{DM}: L^1(\Omega)\to\R\cup\{+\infty\}$.
\end{proof}
The second consequence of Lemma~\ref{lemma:l2l1} is that it makes sense to consider the ground-state energy of the operator $\hat{H}[v]$ for $v\in L^2(\Omega)$, i.e., the map $E : L^2(\Omega) \to \R$ is meaningful. It is also readily seen that Coulomb potentials are in $L^2(\Omega)$. We summarize as a theorem:
\begin{theorem}\label{thm:inabox}
  For $\Omega\subset\R^3$ bounded, $E : L^2(\Omega)\to\R$ and $F_\text{DM} : L^2(\Omega)\to\R\cup\{+\infty\}$ form a skew-conjugate pair of functionals. Moreover, for any $v \in L^2(\Omega)$ there exists a $\rho \in L^2(\Omega)$ such that $E[v] = F_\text{DM}[\rho] + \braket{v,\rho}_2$ (i.e., ${\supdiff} E[v] \neq\emptyset$).
\end{theorem}
\begin{proof}
  We only need to prove existence of a ground-state density for any $v \in L^2(\Omega)$.  For $v\in L^2(\Omega)$, the self-adjoint $N$-electron Hamiltonian associated with the Hamiltonian quadratic form has domain $H^2_{N,0}(\Omega)$~\cite[Theorem 1]{Kato1951}. The Rellich--Kondrachov theorem implies that this Hamiltonian has a compact resolvent, and thus a purely discrete spectrum. In particular, the ground-state energy along with a ground-state eigenvector $\psi \in H^1_{N,0}(\Omega)$ exists for any external potential $v \in L^2(\Omega)$, and thus $\rho_\psi$ is a minimizer for $E[v]=\inf_{\rho} F_\text{DM}[\rho]+\braket{v,\rho}_2$.
\end{proof}
One may ask, why not use the unbounded domain $\Omega = \R^3$ when considering $L^2(\Omega)$ as the density and potential space? In this case convergence in $L^2$ does not imply $L^1$ convergence, and one opens up the possibility that that $F_\text{DM}$ is not lower semicontinuous, so that the Lieb functional $F$ is strictly different from $F_\text{DM}$. This is, at the very least, a conceptual problem. Indeed, from Ref~\cite{Kvaal2015} we have the following:

\begin{theorem} Let $E : L^2(\R^3) \to \R$ be defined by
\[ E[v] = \inf_{\rho \in L^2(\R^3)} \left\{ F_\text{DM}[\rho] + \braket{v,\rho} \right\}, \]
and let $F = E^\vee$. 
\begin{enumerate} 
\item  $E[v] \leq 0$. If $v\geq 0$ almost everyhere, then $E[v] = 0$.
\item $F[0] = 0$, and hence $F_\text{DM}$ is not lower semicontinuous.
\item $0 \in \underline{\partial} F[0]$ and $0 \in \overline{\partial} E[0]$, but there are no ground-state wavefunctions for the Hamiltonian $\hat{H}[0]$.
\end{enumerate}
\end{theorem}

\begin{proof}
  Proof of i): Writing $v = v_+ - v_-$, with $v_{\pm} \in L^2(\R^3)$ almost everywhere positive, we obtain
  \begin{equation}
    E[v] = \inf_\rho \left\{ F_\text{DM}[\rho] + \braket{v_+ ,\rho} - \braket{v_- ,\rho} \right\} \leq
    \inf_\rho \left\{ F_\text{DM}[\rho] + \braket{v_+,\rho} \right\} = E[v_+].
  \end{equation}
  It is therefore sufficient to show that $E[v_+] \leq 0$. 
  
  Let $v \geq 0$ almost everywhere.
  Let $\lambda>0$ be arbitrary and let $\Omega_{k,\lambda}$ with $k\in \N$ be disjoint 
  cubes of side length $\lambda$ such that $\cup_k \Omega_{k,\lambda} =
  \R^3$. Each $\Omega_{k,\lambda}$ can be obtained by translation of $\Omega_{1,\lambda}$.
  Since $v\in L^2(\R^3)$, 
  \[ \int_{\R^3} \!\! v(\mathbf r)^2 \,\mathrm d \mathbf r =\sum_k
  \int_{\Omega_{k,\lambda}} \!\!\!\!v (\mathbf r)^2 \,\mathrm d \mathbf r< +\infty, \]
  implying that $\int_{\Omega_{k,\lambda}} \!\! v(\mathbf r)^2\,\mathrm d \mathbf r \to 0$ as $k\to\infty$.
  Let $\psi \in \mathcal{W}_N$ have smooth components with support in $\Omega_{1,1}^{N}$ so that
  $\rho_\psi \in \mathcal{C}^\infty_c(\R^3)$ has support contained in $\Omega_{1,1}$. By translating
  $\psi$ (denoting the result by $\psi_k$), we can make the support of $\psi_k$ to be inside
  $\Omega_{k,1}^N$ and
  \begin{equation}
    F_\text{DM}(\rho_{\psi_k}) = F_\text{DM}(\rho_\psi) \leq \left\langle \psi \vert T+W \vert \psi \right\rangle
    \equiv \braket{T} + \braket{W},
  \end{equation}
independently of $k$.
We obtain
  \begin{equation}
    E(v) \leq \inf_k \left(\braket{T} + \braket{W} + \int_{\Omega_{k,1}} \!\!\!\!v(\mathbf r) \,\rho_k(\mathbf r) \,\mathrm d \mathbf r \right) = \braket{T} + \braket{W},
  \end{equation}
  where we have used the fact that
  \begin{equation}
    \int_{\Omega_{k,1}} \!\!\!\!v(\mathbf r)\,\rho_k(\mathbf r)\,\mathrm d\mathbf r 
\leq \left(\int_{\Omega_{k,1}}\!\!\!\!v(\mathbf r)^2\, \mathrm d\mathbf r \right)^{1/2} \, \|\rho\|_2 \to 0
  \end{equation}
  as $k\to\infty$.  We now increase the size of the boxes $\Omega_{k,\lambda}$ by varying $\lambda>1$.
  By dilating $\psi$ in the manner 
  \begin{equation}
    \psi(\vec{r}_1,\cdots) \to \lambda^{-3N/2}
    \psi(\lambda^{-1}\vec{r}_1,\cdots),
  \end{equation}
  the support is still inside $\Omega_{1,\lambda}^N$ and the density
  is scaled as $\rho_\psi(\vec{r})\to
  \lambda^{-3}\rho_\psi(\lambda^{-1}\vec{r})$.
We obtain the scaling
  \begin{equation}
    \braket{T} + \braket{W} \to \lambda^{-2}\braket{T} +
    \lambda^{-1}\braket{W}.
  \end{equation}
  By repeating the above argument for $\lambda=1$ and letting
  $\lambda\to\infty$, we obtain $E[v]\leq 0$. On the other hand, $E[v]
  \geq 0$ since the Hamiltonian $\hat{H}[v]$ is positive, yielding $E[v] = 0$.
  Proof of ii): $F_\text{DM}[0]=+\infty$, but $F[0] = \sup_v E[v] = 0$. Thus $F[0] = (F_\text{DM}^\wedge)^\vee[0] < F_\text{DM}[0]$, and $F_\text{DM} \notin \Gamma_0(L^2(\R^3))$.
  Proof of iii): Easy.
\end{proof}

 %
 %
 %
 %
 %
 %

\section{Moreau--Yosida regularization}
\label{sec:my}

In this section we introduce Moreau--Yosida regularization of convex lower-semicontinuous functions over separable Hilbert spaces. The Moreau--Yosida regularization of a convex optimization problem is \emph{invertible}. Hence, we trade a possibly non-smooth optimization problem for a smooth and indeed quite well-behaved one. The invertibility implies that the solution of the exact problem is connected to the solution of the regularized problem. In this exposition, we present elementary proofs of most statements due to their central nature in this chapter. Most proofs are adapted from the excellent monograph~\cite{Bauschke2011} by Bauschke and Combettes, to which the reader is pointed for further details.

\subsection{The Moreau envelope}

In this section,  $(\Hilb,\braket{\cdot,\cdot})$ is a generic separable real Hilbert space. By the usual identification of $\Hilb$ and $\Hilb^*$ and the identification of weak and weak-$*$ topologies, we have $\Gamma_0(\Hilb) = \Gamma_0^*[(\Hilb)^*]$. Our central objects of study is the Moreau envelope ${}^\epsilon f \in \Gamma_0(\Hilb)$ of a convex lower semicontinuous function $f\in \Gamma_0(\Hilb)$ and the associated proximal mapping $\prox_{\epsilon f}:\Hilb\to\Hilb$, introduced by by Moreau in Ref~\cite{Moreau1965}. The name Yosida is also attached to the formalism: In his proof of the Hille--Yosida theorem in Ref.~\cite{Yosida1948} characterizing the generators of strongly continuous one-parameter semigroups, Yosida introduced a certain approximation of the resolvent of maximal monotone operators $A : \Hilb\to 2^\Hilb$. Moreau's Theorem~\cite{Moreau1965,Showalter1991} connects the two concepts in the case where $A = {\subdiff} f$, see Remark~\ref{rem:moreau}.

\begin{definition}[Moreau envelope]\label{def:moreau}
  Let $\epsilon>0$ be given, and let $f :\Hilb \to \R\cup\{+\infty\}$ with nonempty $\dom(f)$. The \emph{Moreau envelope} ${}^\epsilon f : \Hilb \to \R\cup\{\pm\infty\}$ is defined by infimal convolution by the function $x\mapsto (1/2\epsilon)\|x\|^2$,
  \begin{equation}
    {}^\epsilon f[x] = \min_{z\in\Hilb}\left( f[z] + \frac{1}{2\epsilon}\|x-z\|^2 \right) .\label{eq:envelope}
  \end{equation}
  The unique minimizer Eq.~\eqref{eq:envelope} for $\epsilon=1$ is defined as $\prox_{f} x$, i.e.,
  \begin{equation}
    {}^1 f[x] = f[\prox_{ f} x] + \frac{1}{2}\|x-\prox_{f} x\|^2.\label{eq:prox-def}
  \end{equation}
  The map $\prox_f$ is called \emph{the proximal mapping}, and for general $\epsilon>0$ we have
  \begin{equation}
    {}^\epsilon f[x] = f[\prox_{\epsilon f} x] + \frac{1}{2\epsilon}\|x-\prox_{\epsilon f} x\|^2,\label{eq:prox-def-2}
  \end{equation}
  that is, the unique minimizer is in general $\prox_{\epsilon f}x$.
\end{definition}

\begin{proposition}\label{prop:prox}
  Definition~\ref{def:moreau} makes sense, i.e.,  a minimizer of Eq.~\eqref{eq:envelope} exists and is unique.
\end{proposition}
\begin{proof}
    The case $\epsilon=1$ defines $\prox_f x$, and we leave it to the reader to show that the minimizer is $\prox_{\epsilon f}$ in general. Let $x\in \Hilb$ be given, and let $h_{x}[z] := f[x] + (1/2\epsilon)\|z-x\|^2$. Let $\{z_n\}\subset \Hilb$ be a minimizing sequence for the infimal convolution at $x$, i.e., $h_x[z_n]\to {}^\epsilon f[x]$. Such a sequence exists since $\dom(f)\neq\emptyset$. The sequence must be bounded, since $h_x[z_n]\to+\infty$ whenever $\|z_n\|\to+\infty$. By taking a subsequence if necessary, we can assume that $z_n$ converges weakly so some $z_*\in \Hilb$. We note that $h_{x}$ is sequentially weakly lower semicontinuous (any lower semicontinuous convex function is sequentially weakly lower semicontinuous by Mazur's Lemma), and thus
    \[ h_{x}[z_*] \leq \liminf_n h_x[z_{n}] = {}^\epsilon f[x]. \]    Thus, the infimum in the infimal convolution is attained at $z_*$.
    To show uniqueness, we use that strictly convex functions have unique minima: Suppose $z_*'$ is a different minimizer. Since $h_{x}$ is strictly convex, $h_x[\lambda z_* + (1-\lambda)z_*'] < \lambda h_x[z_*] + (1-\lambda)h_x[z_*']$ for $\lambda\in(0,1)$. But then $h_x[\lambda z_* + (1-\lambda) z_*'] < {}^\epsilon f[x]$, a contradiction.
\end{proof}
\begin{remark}
    In the proof of Proposition~\ref{prop:prox}, it is used that $\Hilb$ is a Hilbert space (more precisely, reflexivity is used) when picking a weakly convergent subsequence. Thus, in the general nonreflexive setting, the proximal mapping may not be well-defined.
\end{remark}

\subsection{The proximal mapping}

Some important properties of the proximal mapping are the following:
\begin{proposition}\label{prop:prox-prop}
  Let $f \in\Gamma_0(\Hilb)$ and let $\epsilon>0$ be given. The following holds:
  \begin{enumerate}
    \item\label{prop:prox-prop:1} $p = \prox_{\epsilon f} x \iff$ for all $x' \in \Hilb$ we have
    \begin{equation} \epsilon^{-1}\braket{x'-p,x-p} + f[p] \leq f[x'] \, .\label{eq:prox-vari} \end{equation}
          \item\label{prop:prox-prop:1b}
    $p = \prox_{\epsilon f} x \iff \epsilon^{-1}(x-p) \in {\subdiff} f[p]\,$.
    \item\label{prop:prox-prop:2} $\prox_{\epsilon f} : \Hilb \to \Hilb$ is \emph{firmly nonexpansive} \cite[Section 4]{Bauschke2011}, i.e., for all $x,x'\in \Hilb$,
    \[ \|\prox_{\epsilon f} x - \prox_{\epsilon f} x' \|^2 + \| (\Id - \prox_{\epsilon f}) x - (\Id - \prox_{\epsilon f})x'\|^2 \leq \|x-x'\|^2. \]
    In particular $\prox_{\epsilon f}$ and $\operatorname{Id}-\prox_{\epsilon f}$ are both Lipshitz continuous with constant $1$.
    \item\label{prop:prox-prop:3} If $x \in \dom(f)$ and $\epsilon\to 0^+$, then
    \[ \|x - \prox_{\epsilon f} x \|^2 = O(\epsilon) \, .\]
  \end{enumerate}
\end{proposition}
\begin{proof}
  Proof of \ref{prop:prox-prop:1}: Suppose $p = \prox_{f}x$, and let $x'$ be arbitrary. Set $z=\lambda x' + (1-\lambda)p$. We have
  \[ f[p] = \left(\min_{x'\in\Hilb} f[x'] + \frac{1}{2}\|x-x'\|^2\right) - \frac{1}{2}\|x-p\|^2\, , \]
  so that
  \begin{equation*}
    \begin{split}
      f[p] &\leq f[z] + \frac{1}{2}\|x-z\|^2 - \frac{1}{2}\|x-p\|^2 \\
      & \leq \lambda f[x'] + (1-\lambda)f[p] + \frac{1}{2}\|x - p + \lambda(p - x')\|^2 - \frac{1}{2}\|x-p\|^2 \\
      & = \lambda f[x'] + (1-\lambda)f[p] + \frac{1}{2}\lambda^2\|p-x'\|^2 - \lambda\braket{p-x,p-x'}\,.
    \end{split}
  \end{equation*}
  Hence, for $\lambda > 0$,
  \[ f[p] \leq f[x'] + \frac{1}{2}\lambda \|p-x'\|^2 - \braket{p-x,p-x'}. \]
  Since this holds for all $\lambda\in(0,1)$, we obtain Eq.~\eqref{eq:prox-vari} as $\lambda\to 0$, and substituting $\epsilon f$ for $f$. The converse statement is easy: We rearrange Eq.~\eqref{eq:prox-vari}, to get
  \begin{equation*}
    \begin{split}
      f[p] + \frac{1}{2\epsilon}\|x-p\|^2 &\leq f[x'] + \frac{1}{2\epsilon}\|x-p\|^2 + \epsilon^{-1}\braket{x-p,p-x'} \\
      & \leq f[x'] + \frac{1}{2\epsilon}\|x-p\|^2 + \epsilon^{-1}\braket{p-x,p-x'} + \frac{1}{2\epsilon}\|p-x'\|^2 \\
      & = f[x'] + \frac{1}{2\epsilon}\|x-x'\|^2.
    \end{split}
  \end{equation*}
  Since $x'$ was arbitrary, we conclude that $p=\prox_{\epsilon f}x$.

  Proof of \ref{prop:prox-prop:1b}: Follows directly from \ref{prop:prox-prop:1} and the definition of the subgradient.

  Proof of \ref{prop:prox-prop:2}:
  Assume $p = \prox_{\epsilon f} x$ and $p'=\prox_{\epsilon f} x'$. From \ref{prop:prox-prop:1} we get
 \[ \epsilon^{-1}\braket{p'-p,x-p} + f[p] \leq f[p'] \]
 and
 \[ \epsilon^{-1}\braket{p-p',x'-p'} + f[p'] \leq f[p]\,. \]
 Adding these inequalities and rearranging, we get
 \begin{equation}
   \epsilon^{-1}\braket{p-p',x-p - (x'-p')} \geq 0\,. \label{eq:nonexp}
 \end{equation}
 Subtracting $\|x-x'\|^2/2\epsilon$ from each side and rearranging yields the equivalent condition of firm nonexpansiveness,
 \begin{equation}
   \|p-p'\|^2 + \|(x-p) - (x'-p')\|^2 \leq \|x-x'\|^2, \label{eq:nonexp2}
 \end{equation}
 which also shows that $\prox_{\epsilon f}$ and $\operatorname{Id}-\prox_{\epsilon f}$ are both Lipschitz with constant 1.

  Proof of \ref{prop:prox-prop:3}: It is clear from the definition of the Moreau envelope that ${}^\epsilon f[x] \leq f[x]$. Let $\mu = \sup_{\epsilon>0} {}^\epsilon f[x] \leq f[x] < +\infty$, and thus,
  \[ \forall \epsilon>0, \quad f[\prox_{\epsilon f} x] +
  \frac{1}{2\epsilon} \|x - \prox_{\epsilon f} x\|^2 \leq \mu\,. \]
  Consider the map $h_x[z] := f[z] + \|x-z\|^2/2$. Then for all $\epsilon\in(0,1)$, $\prox_{\epsilon f} x$ is in the sublevel set $\{z \in \Hilb \mid|
   h_x[z] < \mu \}$. This set is bounded. Hence $\nu := \sup_{\epsilon \in (0,1)} \|\prox_{\epsilon f} x\| < +\infty$.
  We next use the fact that any $f\in \Gamma_0(\Hilb)$ poseeses a continuous affine minorant~\cite[Theorem 9.19]{Bauschke2011}, i.e., there exists $\omega \in \Hilb$ and $a\in \R$ such that $\braket{\omega, \cdot} + a \leq f$.
  Therefore,
  \[ \mu \geq \braket{\omega,\prox_{\epsilon f} x} + a + \frac{1}{2\epsilon} \|x - \prox_{\epsilon f} x\|^2 \geq -\nu\|\omega\| + a + \frac{1}{2\epsilon} \|x - \prox_{\epsilon f} x\|^2. \]
  Rearranging gives $\|x - \prox_{\epsilon f} x\|^2 \leq 2\epsilon(\mu + \nu\|\omega\|-a) = O(\epsilon)$ as $\epsilon\to 0^+$.
\end{proof}

The Moreau envelope of a function in $\Gamma_0(\Hilb)$ obtains several nice properties, summerized in the following proposition.
\begin{proposition}\label{prop:my}
  Let $f \in\Gamma_0(\Hilb)$ and let $\epsilon>0$ be given. The following holds:
  \begin{enumerate}
    \item\label{prop:my:1} ${}^\epsilon f \in \Gamma_0(\Hilb)$, and $\dom({}^\epsilon f) = \Hilb$.
    \item\label{prop:my:2} For every $\delta > \epsilon$, and for all $x\in \Hilb$,
    \[ \inf f \leq {}^\delta f[x] \leq {}^\epsilon f[x] \leq f[x]\,. \]
    \item\label{prop:my:3}  For every $x \in \Hilb$, ${}^\epsilon f[x] \to f[x]$ from below as $\epsilon\to 0^+$ (even if $x\neq \dom(f)$).
    \item\label{prop:my:4} ${}^\epsilon f$ is Fr\'echet differentiable, with derivative
    \[ \nabla {}^\epsilon f[x] = \epsilon^{-1}(x - \prox_{\epsilon f}x) \,.\]
    The derivative is Lipschitz continuous with constant $\epsilon^{-1}$.
  \end{enumerate}
\end{proposition}
\begin{proof}
Proof of \ref{prop:my:1}: The domain is already established. We show convexity. With differentiability shown in \ref{prop:my:3}, it follows that ${}^\epsilon f \in \Gamma_0(\Hilb)$. Let $h_{x'}[x] = f[x'] + \|x-x'\|^2/2\epsilon$, which is convex. Thus, $h_{x'}[\lambda x_1 + (1-\lambda)x_2] \leq \lambda h_{x'}[x_1] + (1-\lambda)h_{x'}[x_2]$. Taking the infimum with respect to $x'$ on both sides yields that ${}^\epsilon f$ is convex.

Proof of \ref{prop:my:2}: Easy.

Proof of \ref{prop:my:3}: As in the proof of Proposition~\ref{prop:prox-prop}\ref{prop:prox-prop:3}, set $\mu = \sup_{\epsilon>0} {}^\epsilon f[x]$. From \ref{prop:my:2}, ${}^\epsilon f[x] \to \mu \leq f[x]$ from below as $\epsilon \to 0^+$. Therefore, we can assume that $\mu < +\infty$ and demonstrate that $\lim_{\epsilon\to 0^+} {}^\epsilon f[x] \geq f[x]$. From Proposition~\ref{prop:prox-prop}\ref{prop:prox-prop:3} we get
\[ \lim_{\epsilon\to 0^+} {}^\epsilon f[x] = \lim_{\epsilon\to 0^+} f[\prox_{\epsilon f} x] + \frac{1}{2\epsilon} \|x - \prox_{\epsilon f} x\|^2 \geq \liminf_{\epsilon\to 0^+} f[\prox_{\epsilon f} x] \geq f[x] \,, \]
where we used that $f$ is lower semicontinuous and that $\prox_{\epsilon f} x \to x$.

Proof of \ref{prop:my:4}:
  Let $x,x'\in\Hilb$, $p=\prox_{\epsilon f}x$, $p'=\prox_{\epsilon f}x'$. Using the defintion of $\prox_{\epsilon f}$ and Eq.~\eqref{eq:prox-vari}, we derive two inequalities:
  \begin{equation*}
    \begin{split}
      {}^\epsilon f[x'] - {}^\epsilon f[x] &= f[p'] - f[p] + \frac{1}{2\epsilon} ( \|x'-p'\|^2 - \|x-p\|^2) \\
      & \geq \epsilon^{-1}\braket{p'-p,x-p} + \frac{1}{2\epsilon} ( \|x'-p'\|^2 - \|x-p\|^2) \\
      & = \frac{1}{2\epsilon}\left(\| x' - p' - x + p\|^2 + 2\braket{x'-x,x-p} \right) \\
      & \geq \epsilon^{-1}\braket{x'-x,x-p}.
    \end{split}
  \end{equation*}
  A similar calculation gives the second bound ${}^\epsilon f[x'] - {}^\epsilon f[x] \leq \epsilon^{-1} \braket{x'-x,x'-p'}$.
  Combining these two inequalities gives
  \begin{equation*}
    \begin{split}
    0 &\leq {}^\epsilon f[x'] - {}^\epsilon f[x]  - \epsilon^{-1} \braket{x'-x,x-p} \\&\leq \epsilon^{-1}\braket{x'-x,x'-p'} - \epsilon^{-1} \braket{x'-x,x-p} \\
    &= \epsilon^{-1}\braket{x'-x,(x'-p')-(x-p)}.
  \end{split}
  \end{equation*}
  Using Cauchy--Schwarz and the firm nonexpansiveness condition (Proposition~\ref{prop:prox-prop}\ref{prop:prox-prop:2}), we obtain
  \begin{equation*}
    \begin{split}
    0 &\leq {}^\epsilon f[x'] - {}^\epsilon f[x]  - \epsilon^{-1} \braket{x'-x,x-p} \\&\leq \epsilon^{-1} \|x'-x\|\|(x'-p')-(x-p)\| \leq \epsilon^{-1} \|x'-x\|\left(\|x'-x\|^2 - \|p'-p\|^2\right)^{1/2} \\
    & \leq \epsilon^{-1} \|x'-x\|^2.
  \end{split}
  \end{equation*}
  It follows that
  \begin{equation*}
    \lim_{x'\to x} \|x'-x\|^{-1}( {}^\epsilon f[x'] - {}^\epsilon f[x]  - \epsilon^{-1} \braket{x'-x,x-p}) = 0\,,
  \end{equation*}
  which proves that ${}^\epsilon f$ is Fr\'echet differentiable, with
  \[ \nabla {}^\epsilon f[x] = \epsilon^{-1} (x - \prox_{\epsilon f} x)\,. \]
  Since $\operatorname{Id} - \prox_{\epsilon f}$ has Lipschitz constant 1, $\nabla {}^\epsilon f$ has Lipschitz constant $\epsilon^{-1}$.
\end{proof}

We obtain a simple, but interesting fact from the Fr\'echet derivative from Proposition~\ref{prop:my}\ref{prop:my:4} and the variational characterization pf the proximal mapping in Proposition~\ref{prop:prox-prop}\ref{prop:prox-prop:1b}: While $f \in \Gamma_0(\Hilb)$ may not be differentiable, the gradient of ${}^\epsilon f$ is always a subgradient at $\prox_{\epsilon f} x$.
\begin{corollary}
  Let $f\in \Gamma_0(\Hilb)$, and $\epsilon>0$.
  Then
      $ \nabla {}^\epsilon f[x] \in {\subdiff} f[\prox_{\epsilon f} x] $.
\end{corollary}

\begin{remark}\label{rem:moreau}
  \emph{Moreau's Theorem} \cite{Moreau1965,Showalter1991} states that, for $f \in\Gamma_0(\Hilb)$,
  \[ {}^\epsilon f[x] = f[J_\epsilon x] + \frac{\epsilon}{2}\|A_\epsilon x\|^2 \]
  is convex and Fr\'echet differentiable with $\nabla {}^\epsilon f[x] = {}^\epsilon A$.
  Here, ${}^\epsilon A = \epsilon^{-1}(\Id - {}^\epsilon J)$ is the Yosida approximation to the resolvent ${}^\epsilon J = (\Id + \epsilon A)^{-1}$ of the maximal monotone operator $A=\subdiff f$. Thus, ${}^\epsilon J = \prox_{\epsilon f}$ (which is the only statement in Moreau's Theorem we have not proved), and we also have the compelling identity ${\subdiff} {}^\epsilon f = {}^\epsilon {\subdiff} f$.
\end{remark}

\subsection{Conjugate of the  Moreau envelope}

We next discuss the skew concave conjugate of the Moreau envelope and its properties.

\begin{proposition}\label{prop:my-conjugate}
  Let $f \in \Gamma_0(\Hilb)$, and let $\epsilon>0$. Then
  \begin{equation}
    ({}^\epsilon f)^\wedge[x] = f^\wedge[x] - \frac{1}{2}\epsilon\|x\|^2,
  \end{equation}
  which is strictly (indeed strongly) concave, and
  \begin{equation}
    {\supdiff}({}^\epsilon f)^\wedge[x] = {\supdiff}f^\wedge[x] - \epsilon x . \label{eq:moreau-conj-subdiff}
  \end{equation}
\end{proposition}
\begin{proof}
 Let $g = ({}^\epsilon f)^\wedge$, and let $y \in \Hilb$. We first note that
  \[ {}^\epsilon f[x] =  \inf_{x_1 + x_2 = x} \left\{ f[x_1] + \frac{1}{2\epsilon}\|x_2\|^2 \right\}. \]
  By the definition of the conjugate,
  \begin{equation}
    \begin{split}
      g[y] &= \inf_x \left\{ \braket{y,x} + \inf_{x_1+x_2} f[x_1] + \frac{1}{2\epsilon}\|x_2\|^2  \right\} \\
      & = \inf_x \inf_{x_1+x_2=x} \left\{ \braket{y,x} +  f[x_1] + \frac{1}{2\epsilon}\|x_2\|^2  \right\} \\
      & = \inf_{x_1} \inf_{x_2}\left\{ \braket{y,x_1} + f[x_1] + \braket{y,x_2} + \frac{1}{2\epsilon}\|x_2\|^2 \right\} \\
      & = f^\wedge[y] - \frac{\epsilon}{2}\|y\|^2,
    \end{split}
  \end{equation}
  where we have used that if $\varphi[x]=\|x\|^2/2\epsilon$, then $\varphi^\wedge[y] = - \epsilon \|y\|^2/2$, which is left as an exercise.

  To establish Eq.~\eqref{eq:moreau-conj-subdiff}, we appeal to \cite[Theorem 5.38]{VanTiel1984}, which, in the Hilbert space setting, states that if $f_1,f_2\in\Gamma_0(\Hilb)$ and there is a point in $\dom(f_1)\cap\dom(f_2)$ where $f_1$ is continuous, then ${\subdiff}(f_1+f_2)[x] = {\subdiff} f_1[x] + {\subdiff} f_2[x]$ for every $x\in \Hilb$. (The inclusion ${\subdiff}f_1[x]+{\subdiff}f_2[x]\subset{\subdiff}(f_1+f_2)[x]$ is easy to prove, but the converse inclusion does not hold in general.) Our result is established by noting that $\varphi,-f^\wedge \in -\Gamma_0(\Hilb)$, and that $\varphi$ is everywhere continuous.
\end{proof}

\begin{remark}
    Proposition~\ref{prop:my-conjugate} implies that the Moreau--Yosida regularization is lossless, i.e., for any $\epsilon>0$
    \[ f = \left(({}^\epsilon f)^\wedge + \frac{\epsilon}{2}\|\cdot\|^2\right)^\vee, \]
    an explicit formula for the inverse of the regularization. 
\end{remark}

 %
 %
 %
 %
 %
 %
 %
 %
 %
 %
 %
 %
 %
 %
 %
 %
 %

\section{Moreau--Yosida regularized exact DFT}
\label{sec:my-dft}

We apply Moreau--Yosida regularization to exact DFT in a box domain, as outlined in Section~\ref{sec:exact-dft}. The treatment in this section closely follows that of Ref.~\cite{Kvaal2014}. We describe Moreau--Yosida Kohn--Sham (MYKS) theory, and set up a basic self-consistent field (SCF) iteration, an abstract algorithm for the solution of the Kohn-Sham problem. We describe the Moreau--Yosida Kohn--Sham Optimal Damping Algorithm (MYKSODA) algorithm, and prove a weak convergence result from Ref.~\cite{Laestadius2018b}.

In this section, we assume $\Omega = [-\ell,\ell]^3 \subset \R$ is a finite box, and we let $\Hilb = L^2(\Omega)$ throughout. We omit all specifications of $\Omega$ in symbols like $\mathcal{D}^N(\Omega)$ for brevity. We let the Moreau--Yosida parameter $\epsilon>0$ be fixed unless otherwise stated.

\subsection{Regularized universal and ground-state energy functionals}

We consider the Moreau envelope ${}^\epsilon F : \Hilb \to \R$ and its concave skew conjugate. Thus, the central functionals of this section are
\begin{align}
   {}^\epsilon F[\rho] &= F[\prox_{\epsilon F}\rho] + \frac{1}{2\epsilon} \|\rho - \prox_{\epsilon F}\rho\|^2, \\
   {}^\epsilon E[v] & := E[v] - \frac{1}{2}\epsilon \|v\|^2.
\end{align}
These are related via skew conjugation,
\begin{align}
  {}^\epsilon E[v] &= \min_{\rho\in\Hilb} {}^\epsilon F[\rho] + \braket{v,\rho}, \\
  {}^\epsilon F[\rho] &= \max_{v\in\Hilb} {}^\epsilon E[v] - \braket{v,\rho}.
\end{align}
The original infimum and supremum are a minimum and a maximum, since any $v\in\Hilb$ has a ground state by Theorem~\ref{thm:inabox}, and since ${}^\epsilon E$ is strongly concave.
The map ${}^\epsilon E$ is defined \emph{not} as the Moreau envelope of $E$, but is instead related to the unregularized energy $E$ by an explicit and easy-to-evaluate function of $v \in \Hilb$. By Proposition~\ref{prop:my}, ${}^\epsilon F$ is everywhere Fr\'echet differentiable with $\nabla {}^\epsilon F[\rho] = \epsilon^{-1}(\rho - \prox_{\epsilon F}\rho) \in {\subdiff} F[\prox_{\epsilon F}\rho]$. Thus $\prox_{\epsilon F}\rho \in \mathcal{B}_N$. The converse is also true, which gives the remarkable fact that $\mathcal{B}_N$ is the range of the proximal mapping, so that the $v$-representability problem, i.e., the problem of characterizing those $\rho\in\mathcal{D}^N$ that are ground-state densities of a given $v\in \Hilb$, is traded for the problem of evaluating ${}^\epsilon F$, since differentiating this function is not an issue.
\begin{proposition}
  The proximal mappig $\prox_{\epsilon F} : \Hilb \to \mathcal{B}_N$ is onto.
\end{proposition}
\begin{proof}
    That $\prox_{\epsilon F}^{-1} = \Id + \epsilon {\subdiff} F$ comes directly from Proposition~\ref{prop:prox-prop}\ref{prop:prox-prop:1b}.
    By definition of range and domain, $\ran A = \dom A^{-1}$ for any set-valued operator $A : \Hilb \to 2^\Hilb$~\cite[Definition 12.23]{Bauschke2011}, it follows immediately that $\ran( \prox_{\epsilon F}) = \dom (\Id + {\subdiff}F) = \mathcal{B}_N$. \hilite{PROOF COULD AVOID DEF OF INVERSE BY SPELLING OUT...}
\end{proof}
Having established that $\ran(\prox_{\epsilon F}) = \mathcal{B}_N$, we ask what is the preimage of some $\rho_\epsilon \in \mathcal{B}_N$?
The Hohenberg--Kohn Conjecture states that ${\subdiff} F[\rho_0]$ is unique up to an additive constant. Assuming the conjecture to hold, we obtain an interesting property of the proximal mapping for a fixed $\epsilon>0$.
\begin{proposition}
  Let $\rho_0 \in \mathcal{B}_N$ be the ground-state density of $v\in \Hilb$ (which exists since any $v$ has a ground-state in the box-truncated setting). Assume Conjecture~\ref{conj:hk}, i.e.,
  \[ {\subdiff} F[\rho_0] = \{ -v + \mu \mid \mu \in \R \}\, . \]
  Then, $\rho \in \prox_{\epsilon F}^{-1}\rho_0$ is unique up to a constant shift, i.e.,
  \[ \prox_{\epsilon F}^{-1}[\rho_0] = \{ \rho + \nu \mid \nu \in \R  \}\,.\]
\end{proposition}
\begin{proof}
    Any potential for which $\prox_{\epsilon F}\rho=\rho_0$ is a ground state is on the form $v = -\epsilon^{-1}(\rho - \rho_0) + \mu$, with $\mu\in\R$, and $\rho \in \prox_{\epsilon F}^{-1}[\rho_0]$. Rearranging, $\rho = -\epsilon v + \rho_0 - \epsilon \mu$, from which the result follows.
\end{proof}

We can view $\prox_{\epsilon F}$ as a kind of nonlinear ``projection'', that takes a $\rho \in \Hilb$ and maps it to an (ensemble) $v$-representable density $\rho_\epsilon = \prox_{\epsilon F}\rho$. Moreover, any (ensemble) $v$-representable density can be reached from some $\rho \in \Hilb$. However, it is not true in general that $\prox_{\epsilon F}^{\circ 2} = \prox_{\epsilon F}$, so it is not a true projection.

The fact that $\dom({}^{\epsilon} F)=\Hilb$ may seem like a severe shortcoming of the formalism, as $\Hilb$ contains mostly ``unphysical densitites''. On the other hand, for a given $\rho \in \Hilb$, it is $\prox_{\epsilon F}\rho$ which is the \emph{physical} density. We therefore denote general elements of $\Hilb$ as ``quasidensities''.
\begin{definition}
  Let $\rho \in \Hilb$ be given. We define the \emph{proximal density} $\rho_\epsilon[\rho] := \prox_{\epsilon F}\rho$, the resolvent of ${\subdiff} F$, and the \emph{proximal potential} $v_\epsilon[\rho] := -{}^\epsilon \subdiff F[\rho] = - \epsilon^{-1}(\Id - \prox_{\epsilon F})\rho$, with the Yosida approximation ${}^\epsilon \subdiff F = \nabla {}^\epsilon F$ to the resolvent of ${\subdiff} F$, cf.~Remark~\ref{rem:moreau}.
\end{definition}

\begin{remark}\label{rem:duality}
  The proximal density and potential satisfies
  \[ -v_\epsilon[\rho] \in {\subdiff} F[\rho_\epsilon[\rho]] \iff \rho_\epsilon[\rho] \in {\supdiff} E[v_\epsilon[\rho]]\,, \]
  demonstrating how the proximal mapping generates a $v$-representable density \emph{and} a corresponding potential via the Yosida approximation.
\end{remark}

The regularized functionals are approximations to the exact functionals in the sense of Proposition~\ref{prop:my}. For example, for every quasidensity $\rho \in \mathcal{D}^N$, ${}^\epsilon F[\rho] \to F[\rho]$ from below as $\epsilon\to 0^+$. Moreover, from Proposition~\ref{prop:prox-prop}\ref{prop:prox-prop:3}, we get $\rho_\epsilon \to \rho$. On the other hand, if $\rho \notin \mathcal{D}^N$, then ${}^\epsilon F[\rho] \to +\infty$.

\subsection{Regularized Kohn--Sham theory}

Like in traditional Kohn--Sham theory, we set up a fictitious non-interacting system with an effective potential $v_\text{eff} \in \Hilb$ to be determined in a self-consistent manner. However, whereas in traditional Kohn--Sham the interacting and non-interacting systems are required to have the same \emph{physical} densities, regularized Kohn--Sham theory requires the systems to have the same \emph{quasidensities}. The Moreau--Yosida regularization resolves the two major issues with traditional Kohn--Sham theory: First, that that $\rho$ needs to be both interacting and non-interacting $v$-representable, resolved by $\rho\in\Hilb$ being $v$-representable for any interaction strength in the Moreau--Yosida regularized formulation. Second, that $F$ is not at all differentiable, which makes the set-up of the self-consistent field equations non-rigorous, resolved by ${}^\epsilon F$ being continuously Fr\'echet differentiable by Proposition~\ref{prop:my}\ref{prop:my:4}.

\subsubsection{Adiabatic connection}

To prepare for Kohn--Sham theory, we briefly discuss the \emph{adiabatic connection}, which relates the interacting $N$-electron problem to a non-interacting one. We introduce a connection parameter $\lambda \in [0,1]$, and consider the modified energy functional on $\mathcal{W}_N$ given by
\[ \mathcal{E}^\lambda_0[\psi] = \frac{1}{2}\braket{\nabla\psi,\nabla\psi} + \lambda \braket{\psi,W\psi}, \]
implying that the ground-state energy is a function of $\lambda$ as well as the potential,
\begin{equation}
  E^\lambda[v] := \inf\{ \mathcal{E}_0^\lambda[\psi] + \mathcal{V}[\psi] \mid \psi \in \mathcal{W}_N \}.
\end{equation}
At $\lambda=0$, the problem is ``solvable'' in the sense that the $N$-electron Schdrödinger equation becomes separable, a key motivation behind Kohn--Sham theory.
Similarly, we obtain a $\lambda$-dependent universal functional,
\begin{equation}
  F^\lambda[\rho] = \inf_{\gamma\mapsto \rho} \Tr(\hat{H}[0]^\lambda \gamma)\,.
\end{equation}
The connection to spectral theory of self-adjoint operators in Section~\ref{sec:exact-dft-1} dos not depend on choosing $\lambda = 1$. Indeed, any $\lambda \in \R$ is acceptable. Moreover, the present discussion is valid for both the case $\Omega = \R^3$ and $\Omega$ a finite box.

\begin{proposition}
  Fix $v \in X^*$ and $\rho \in \mathcal{D}^N$. The maps from $\R$ to $\R$ defined by $\lambda \mapsto E^\lambda[v]$ and $\lambda \mapsto F^\lambda[\rho]$  are everywhere finite, concave, left and right differentiable, and almost everywhere differentiable.
\end{proposition}
\begin{proof}
  We need to show concavity of $\lambda \mapsto E^\lambda[v]$. This follows from
  $E^\lambda[v] = \inf_{\gamma \in \mathcal{D}^N_\text{op}} \{ A_\gamma + \lambda B_\gamma\}$
  with $A_\gamma = \Tr((T + V)\gamma)$ and $B_\gamma = \Tr(W\gamma)$, i.e., a pointwise infimum of a family of affine functions $\lambda \mapsto \Tr(\hat{H}_0[v] + \lambda \Tr(\hat{W}\gamma)$. Concavity of $\lambda\mapsto F^\lambda[\rho]$ follows immediately. It is clear that the both the maps are finite for all $\lambda\in\R$. By \cite[Theorem 1.6]{VanTiel1984}, concave functions are everywhere left and right differentiable on the interior of their domains, and by Theorem 1.8 in the same reference, almost everywhere differentiable.
\end{proof}

One immediate complication that arises in the adiabatic connection is that the set of ensemble $v$-representable densities may change with $\lambda$, i.e., the set-valued map $\lambda \mapsto \mathcal{B}_N^\lambda \subset \mathcal{D}^N$ is non-trivial. There are no known results that relate the different $\mathcal{B}_N^\lambda$.

Turning to the Moreau--Yosida regularization in the domain $\Omega = [-\ell,\ell]^3$, we obtain $\lambda$-dependent functionals ${}^\epsilon E^\lambda : \Hilb \to \R$ and ${}^\epsilon F^\lambda : \Hilb \to \R$. Fix $\hat{\rho} \in \Hilb$, a target density. For any $\lambda\in\R$, there exists a unique proximal potential,
\[ v_\epsilon^\lambda [\hat{\rho}] = -\nabla {}^\epsilon F^\lambda[\hat{\rho}]\,. \]
The regularized ground-state energy of that potential is
\[ {}^\epsilon E^\lambda[v_\epsilon^\lambda] = {}^\epsilon F^\lambda[\hat{\rho}] + \braket{v^\lambda_\epsilon,\hat{\rho}}, \]
related to the \emph{unregularized} energy as
\[ E^\lambda[v_\epsilon^\lambda] = {}^\epsilon E^\lambda[v_\epsilon^\lambda] + \frac{\epsilon}{2}\|v_\epsilon^\lambda\|^2. \]
Associated with $\hat{\rho}\in\Hilb$ is also the proximal density $\rho_\epsilon^\lambda[\hat{\rho}] = \prox_{\epsilon F^\lambda}[\hat{\rho}]$,
being the true ground-state density of the proximal potential,
\[ E^\lambda[v^\lambda_\epsilon] = F^\lambda[\rho^\lambda_\epsilon] + \braket{v^\lambda_\epsilon,\rho^\lambda_\epsilon}. \]
This is the adiabatic connection in the Moreau--Yosida formulation of DFT.  However, we leave open the perhaps important question of the regularity properties of the maps $(\hat{\rho},\lambda) \mapsto v^\lambda_\epsilon[\hat{\rho}]$ and $(\hat{\rho},\lambda) \mapsto \rho^\lambda_\epsilon[\hat{\rho}]$.

\subsubsection{Kohn--Sham decomposition}

Let $v_\text{ext} \in \Hilb$ be the external potential for which we wish to determine $E^1[v_\text{ext}]$ and its corresponding ground-state density, known to exist by Theorem~\ref{thm:inabox}. (If the ground state of $\hat{H}[v_\text{ext}]$ is degenerate, there will be a convex set of solutions.) By Proposition~\ref{prop:my-conjugate}, this problem is in a simple manner connected to the Moreau--Yosida regularized problem by Proposition~\ref{prop:my-conjugate},
\[ {\supdiff} E^1[v_\text{ext}] = {\supdiff} {}^\epsilon E^1[v_\text{ext}] + \epsilon v_\text{ext}\,, \]
and hence,
\[ \rho \in {\supdiff} {}^\epsilon E^1[v_\text{ext}] \iff \rho_\epsilon^1 = \rho - \epsilon v_\text{ext} \in {\supdiff} E^1[v_\text{ext}]. \]
We connect this problem at $\lambda=1$ to the non-interacting problem at $\lambda=0$, where now the quasidensity is known, whereas the corresponding potential $v_\text{eff} \in \Hilb$ is unknown,
\[ \rho \in {\supdiff} {}^\epsilon E^0[v_\text{eff}] \iff \rho_\epsilon^0 = \rho - \epsilon v_\text{eff} \in {\supdiff} E^0[v_\text{eff}]. \]
We know that the potential $v_\text{eff}\in\Hilb$ exists, since ${}^\epsilon F^0$ is differentiable at $\rho$.
By Proposition~\ref{prop:my}, we have the stationary conditions
\begin{align}
  -v_\text{ext} &= \nabla {}^\epsilon F^1[\rho]\,, \label{eq:myks-stat-1} \\
  -v_\text{eff} &= \nabla {}^\epsilon F^0[\rho]\,. \label{eq:myks-stat-0}
\end{align}
We introduce the \emph{Hartree-exchange correlation energy} ${}^\epsilon E_\text{Hxc}[\rho] := {}^\epsilon F^1[\rho] - {}^\epsilon F^0[\rho]$, which is continuously differentiable with Lipschitz continuous gradient ${}^\epsilon v_\text{Hxc}[\rho] := \nabla {}^\epsilon E_{Hxc}[\rho]$, denoted the Hartree-exchange correlation potential. The condition~\eqref{eq:myks-stat-1} becomes
\begin{equation}
  v_\text{eff} = v_\text{ext} + {}^\epsilon v_\text{Hxc}[\rho]\,. \tag{\ref{eq:myks-stat-1}'}\label{eq:myks-stat-1'}
\end{equation}
The idea is now that it is $E_\text{Hxc}$, and not ${}^\epsilon F^1$, which is available, at least as some kind of approximation. Equations~\eqref{eq:myks-stat-0} and \eqref{eq:myks-stat-1'} is a self-consistent field problem where the density $\rho$ and potential $v_\text{eff}$ are the unknowns.

This motivates the following abstract algorithm:
\begin{algorithm}[Basic MYKS-SCF iteration scheme]\label{algo:1} \leavevmode
  Fix $\epsilon > 0$.
  \begin{enumerate}[label={\arabic*)}]
    \item Choose an initial guess $\rho_0 \in \Hilb$.
    \item For $i = 0,1,\ldots$, iterate:
    \begin{enumerate}[label={\arabic*)}]
    \item Set $v_{\text{eff},i+1} = v_\text{ext} +  {}^\epsilon v_{\text{Hxc}}[\rho_i]$.
    \item Solve the unregularized nonineracting problem $\tilde{\rho}_{i+1} \in {\supdiff} E^0[v_{\text{eff},i+1}]$, and set $\rho_{i+1} = \tilde{\rho}_{i+1} - \epsilon v_{i+1}$. 
    \item If $v_{\text{eff},i+1} = v_\text{ext} + {}^\epsilon v_\text{Hxc}[\rho_{i+1}]$, terminate.
  \end{enumerate}
  \item Return $\rho^1_\epsilon = \rho_{i+1} + \epsilon v_\text{ext}$ as the (physical) ground-state density, and $E^1[v_\text{ext}] = F^0[\rho_{i+1}] + {}^\epsilon E_\text{Hxc}[\rho_{i+1}] + \epsilon \|v_\text{ext}\|^2/2$ as the ground-state energy.
\end{enumerate}
\end{algorithm}
\begin{remark}
    In Step 2.1 of Algorithm~\ref{algo:1}, the problem $\tilde{\rho}_{i+1} \in {\supdiff} E^0[v_{\text{eff},i+1}]$ is equivalent to solving the eigenvalue equation of a one-electron operator $\hat{h} = -\nabla^2/2 + v_{\text{eff},i+1}$. By the Rellich--Kondrachov theorem, this operator has domain $H^2_0(\Omega)$, a compact resolvent and hence a purely discrete spectrum.  The first $N$ eigenfunctions $\varphi_k\in H^2_0(\Omega)$ define a solution $\tilde{\rho}_{i+1} = \sum_{k=1}^N |\varphi_k|^2$. Any degeneracies in the spectrum lead a finite number of solutions, the convex hull of which is ${\supdiff} E^0[v_{\text{eff},i+1}]$. In particular, the superdifferential is always nonempty, so that the noninteracting problem can always be solved.
\end{remark}
\begin{remark}
    The formal limit $\epsilon \to 0^+$ in Algorithm~\ref{algo:1} gives the traditional Kohn--Sham self-consistent field iterations for exact DFT. Of course, the Hartree-exchange correlation potential is not rigorously defined in this case.
\end{remark}

\subsection{Optimal-damping algorithm}

Simple SCF iterations often suffer from bad convergence properties, with phenomena like asymptotic oscillations. This particular behavior can be traced to the fact that each SCF iteration takes too large a step, so that one ``jumps over'' a minimum.
The optimal-damping algorithm (ODA) for the Hartree--Fock SCF iterations were introduced by Cancès and Le Bris in Refs.~\cite{Cances2000,Cances2000b}, in which the ``bare'' SCF density update is damped by a factor in $(0,1]$. In this section, we apply the idea of optimal damping for extended Kohn--Sham iterations from Ref.~\cite{Cances2001} to modify Algorithm~\ref{algo:1}, leading to the Moreau--Yosida regularized Kohn--Sham optimal-damping algorithm (MYKSODA). This algorithm was introduced in Ref.~\cite{Laestadius2018b}, where the convergence of the energy to an upper bound, Theorem~\ref{thm:myksoda}, was proven. A stronger convergence proof for a finite-dimensional setting, Theorem~\ref{thm:myksoda2}, was introduced in Refs.~\cite{Penz2019,Penz2019erratum}, and we here present the full proof with some additional details and corrections.

The ODA modification of Algorithm~\ref{algo:1} consists of a relaxation of the density step, i.e., not setting $\rho_{i+1}=\rho'_{i+1}$, but instead $\rho_{i+1} = \rho_i + t_i\rho'_{i+1}$ for some $t_i \in (0,1]$. In Theorem~\ref{thm:myksoda2}, the step lengths $t_i$ are actually not constrained, but convergence is still guaranteed.
\begin{algorithm}[MYKSODA]\label{algo:2} \leavevmode
  Fix $\epsilon > 0$.
  \begin{enumerate}[label={\arabic*)}]
    \item Set $v_{\text{eff},0} = v_\text{ext}$, and $\rho_0 \in {\supdiff} {}^\epsilon E^0[v_\text{ext}]$.
    \item For $i = 0,1,\ldots$, iterate:
    \begin{enumerate}[label={\arabic*)}]
    \item Set $v_{\text{eff},i+1} = v_\text{ext} +  {}^\epsilon v_{\text{Hxc}}[\rho_i]$.
    \item Solve for $\tilde{\rho}_{i+1} \in {\subdiff} E^0[v_{\text{eff},i+1}]$, and set $\rho_{i+1}' = \tilde{\rho}_{i+1} - \epsilon v_{\text{eff},i+1}$.
    \item Choose $t_i \in (0,1]$ such that $\rho_{i+1} = \rho_i + t_i(\rho_{i+1}' - \rho_i)$ satisfies
    \begin{equation} \braket{\nabla {}^\epsilon F[\rho_{i+1}] + v_\text{ext}, \rho_{i+1}' - \rho_i } \leq 0 \,.\label{eq:oda} \end{equation}
    \item If $v_{\text{eff},i+1} = v_\text{ext} + {}^\epsilon v_\text{Hxc}[\rho_{i+1}]$, terminate.
  \end{enumerate}
  \item Return $\rho^1_\epsilon = \rho_{i+1} + \epsilon v_\text{ext}$ as the (physical) ground-state density, and $E^1[v_\text{ext}] = F^0[\rho_{i+1}] + {}^\epsilon E_\text{Hxc}[\rho_{i+1}] + \epsilon \|v_\text{ext}\|^2/2$ as the ground-state energy.
\end{enumerate}
\end{algorithm}
The MYKSODA algorithm is weakly convergent in the sense that the energy converges to an upper bound:
\begin{theorem}\label{thm:myksoda}
  In Algorithm~\ref{algo:2}, the sequence of energy estimates
  \[ e_i := {}^\epsilon F[\rho_i] + \braket{v_\text{ext},\rho_i} \]
  is strictly descending and hence convergent. Denoting the limit $e[v_\text{ext}]$, we have
  \[ e[v_\text{ext}] + \epsilon \|v_\text{ext}\|^2/2 \geq E^1[v_\text{ext}]\,, \]
  that is, an upper bound property of the computed energy limit.
\end{theorem}
\begin{proof}
  It is clear that step 2.2 produces $\rho'_{i+1}\in \underline{\partial}{}^\epsilon E^0[v_{\text{eff},i+1}]$. We need to check that step 2.3 is well-defined. To that end, consider the function $G[\rho] := {}^\epsilon F^1[\rho] + \braket{v_\text{ext},\rho}$, which is to be minimized by the iterations. We note that step 1.1 can be rewritten $v_{\text{eff},i+1} + \nabla {}^\epsilon F^0[\rho_i] = v_\text{ext} + \nabla {}^\epsilon F^1[\rho_i]$. The directional derivative of $G$ at $\rho$ in the direction $\Delta\rho_i := \rho_{i+1}'-\rho_i$ is $G'[\rho;\Delta\rho_i] = \braket{\nabla {}^\epsilon F^1[\rho] + v_\text{ext},\rho_{i+1}'-\rho_i}$, which gives
  \begin{equation*}
    \begin{split}
   G'[\rho_i;\rho_{i+1}'-\rho_i] &= 
    \braket{\nabla {}^\epsilon F^0[\rho_i] + v_{\text{eff},i+1}, \rho'_{i+1} - \rho_i}.
    \end{split}
 \end{equation*}
 If the termination criterion was met at the previous iteration, $G'[\rho_i;\rho_{i+1}'-\rho_i]=0$, and step 2.3 will not take place. Otherwise, since $\rho_i \in \supdiff {}^\epsilon E^0[-\nabla {}^\epsilon F^0[\rho_i]]$ (see Remark~\ref{rem:duality}), write
 \[ \sigma_{i+1} := \rho'_{i+1} + \epsilon v_{\text{eff},i+1} \in \supdiff E^0[v_{\text{eff},i+1}], \]
 \[ \tau_i:= \rho_i + \epsilon\nabla{}^\epsilon F^0[\rho_i] \in \supdiff E^0[-\nabla{}^\epsilon F[\rho_i]]. \]
This allows us to rewrite the directional derivative as
 \[ G'[\rho_i;\rho_{i+1}'-\rho_i] = \braket{v_{\text{eff},i+1} + \nabla {}^\epsilon F^0[\rho_i], \sigma_{i+1} - \tau_{i} } - \epsilon\| v_{\text{eff},i+1} + \nabla {}^\epsilon F^0[\rho_i]\|^2. \]
 By monotonicity of the superdifferential (Prop.~\ref{prop:subdiff}\ref{prop:subdiff:1}), the bracket is always nonpositive, which by step 2.1 in Algorithm~\ref{algo:2} gives
 \begin{equation}\label{eq:grad-G-mono} G'[\rho_i;\Delta\rho_i] \leq -\epsilon^{-1} \| v_{\text{eff},i+1} + \nabla {}^\epsilon F^0[\rho_i]\|^2 = - \epsilon\| v_{\text{ext}} + \nabla {}^\epsilon F^1[\rho_i]\|^2 . \end{equation}
 Since now $G[\rho_i]$ is strictly decreasing and convex in the direction $\Delta\rho_i$ ($t\mapsto G[\rho_i + t\Delta\rho_i]$ is \emph{strongly} convex at $t=0$), it follows that there is a maximum step length $t_i > 0$ such that $e_{i+1} = G[\rho_{i+1}] < G[\rho_i] = e_i$, and such that Eq.~\eqref{eq:oda} holds, i.e., $G'[\rho_{i+1},\rho_{i+1}'-\rho_i] \leq 0$.
 Since $e_i \geq \inf_\rho G[\rho] = {}^\epsilon E^1[v_\text{ext}]$ for all $i$, it follows that the monotonically decreasing sequence $e_i$ converges to an upper bound of ${}^\epsilon E^1[v_\text{ext}] = E^1[v_\text{ext}] - \epsilon \|v_\text{ext}\|^2/2$.
\end{proof}



In the above, the step lengths $t_i \in (0,1]$ were not specified. It is clear that any stronger convergence properties of the algorithm depend crucially on these. The final result of this section, paraphrased from Refs.~\cite{Penz2019,Penz2019erratum}, exploits the definition of the Moreau--Yosida regularization at every step. The key idea is the following interpretation of Prop.~\ref{prop:my}\ref{prop:my:4}: At every $x \in\mathcal{H}$, ${}^\epsilon f$ is tangent to the regularization parabola $h_x(z) := f(\prox_{\epsilon f} x) + \|z-\prox_{\epsilon f}x\|^2/(2\epsilon)$, i.e., evaluated at $x$, both have gradient $\epsilon^{-1}(x -\prox_{\epsilon f} x)$ and their graphs are touching at $x$. Moreover, by definition ${}^\epsilon f \leq h_z$. The step sizes $t_i$ are now chosen such that $\rho_{i+1}$ minimizes the regularization parabola associated with $G[\rho_i]$ along the search direction $\Delta\rho_i$. However, one must assume that $\mathcal{H}$ is finite-dimensional for this to guarantee convergence. A finite dimensional $\mathcal{H}$ arises, e.g., when the $N$-electron Hilbert space is discretized by a finite-dimensional basis generated by a finite set of single-particle functions, and does not change the properties of $E$ or $F$ in any way relevant for the current discussion.
\begin{theorem}\label{thm:myksoda2}
    Assume that $\mathcal{H}$ is finite dimensional. Then
    there exists a sequence of step lengths $\{t_i\}_i$ for Algorithm~\ref{algo:2} such that $\{\rho_i\}$  has energies $e_{i} = {}^\epsilon F[\rho_{i}] + \braket{v_\text{ext},\rho_{i}}$ converging to the exact result ${}^\epsilon E[v_\text{ext}]$. Moreover, $\{\rho_i\}$
    is the union of convergent subsequences $\rho_{i_n} \to \sigma \in \mathcal{H}$, such that $\sigma \in \supdiff {}^\epsilon E[v_\text{ext}]$ is an exact ground-state quasidensity, and $v_{\text{eff},i} \to v_\text{ext} + {}^\epsilon v_{\text{Hxc}}[\sigma]$.
\end{theorem}
\begin{proof}
Let $p_i := \prox_{\epsilon F}\rho_i$. Rewritten in terms of this proximal point, we have
\begin{equation}\label{eq:G}
    \begin{split}
        G[\rho_i] &= F[p_i] + \braket{v_\ext,\rho_i} + \frac{1}{2\epsilon}\|\rho_i-p_i\|^2 \\
        &= F[p_i] + \braket{v_\ext,p_i} - \frac{\epsilon}{2}\|v_\ext\|^2 + \frac{1}{2\epsilon}\|\rho_i-p_i+\epsilon v_\ext\|^2.
    \end{split}
\end{equation}  
The regularization parabola associated with $G[\rho]$ is similarly obtained, noting that the linear term $\braket{v_\ext,\rho}$ of $G$ is easily absorbed,
\begin{equation}\label{eq:h} h_{\rho_i}[\sigma] := F[p_i] + \braket{v_\ext,p_i} - \frac{\epsilon}{2}\|v_\ext\|^2 + \frac{1}{2\epsilon}\|\sigma-p_i+\epsilon v_\ext\|^2. \end{equation}
We note that the global minimum of $h_{\rho_i}$ is at $p_* := p_i - \epsilon v_\ext$, and that $\nabla h_{\rho_i}[\rho_i] = \nabla G[\rho_i]$, by the remarks preceding the theorem.
Let $f(t) = G[\rho_i + t \Delta\rho_i]$, which by Eq.~\eqref{eq:grad-G-mono} satisfies $f'(0)<0$. At all $t$, $f(t) \leq g(t) := h_{\rho_i}[\rho_i + t\Delta \rho]$. Its derivative is
\[ g'(t) = \epsilon^{-1}\braket{\rho_i - p_* + t\Delta \rho,\Delta \rho}, \quad g'(0)<0. \]
Define $t_i>0$ uniquely (and hence $\rho_{i+1}=\rho_i + t_i\Delta\rho_i$) by the condition $g'(t_0)=0$, which gives 
\[ t_i = - \|\Delta\rho_i\|^{-2}\braket{\rho_i - p_*,\Delta \rho}, \]
and the  equivalent orthogonality condition $\braket{\rho_{i+1}-p_*,\rho_{i+1}-\rho_i}=0$. Geometrically, $\rho_{i+1}-\rho_i$ is the orthogonal projection of $\rho_*-\rho_i$ onto the search line. A brief calculation gives
\begin{equation} \label{eq:step-length}
    \begin{split}
     \braket{\nabla G[\rho_i],\Delta\rho_i} &= g'(0) = t_i^{-1} \braket{\nabla{}^\epsilon F[\rho_i] + v_\ext,\rho_{i+1}-\rho_i} = 
     \\ &= t_i^{-1}\epsilon^{-1}\braket{\rho_i - \rho_{i+1} + \rho_{i+1} - p_*,\rho_{i+1}-\rho_i} \\ &= t_i^{-1}\epsilon^{-1}\braket{\rho_i-\rho_{i+1},\rho_{i+1}-\rho_i}  \\ &= -t_i\epsilon^{-1}\|\Delta\rho_i\|\|\rho_{i+1}-\rho_i\| = - \frac{\|\Delta\rho_i\|^2t_i^2}{\epsilon}.
\end{split}
\end{equation}
Let $m_i := h_{\rho_i}[\rho_{i+1}]$ be the minimum of the parabola section. We have $e_{i+1} \leq m_i$. Subtracting Eq.~\eqref{eq:h} with $\sigma = \rho_{i+1}$ from Eq.~\eqref{eq:G} and using orthogonality, we obtain
\[
\frac{t_i^2\|\Delta\rho_i\|^2}{2\epsilon} =  \frac{1}{2\epsilon}\|\rho_{i+1} - \rho_i\|^2 = e_i - m_i \leq e_i - e_{i+1} \to 0.
\]
We now show that $\rho_i$ converges. By assumption, $\dom(F)$  is bounded in $\mathcal{H}$. It is easy to see that ${}^\epsilon F[\rho] = O(\|\rho\|^2)$ as $\|\rho\|\to+\infty$. Since $e_i = {}^\epsilon F[\rho_i] + \braket{v_\ext,\rho_i} \leq e_1$ for all $i$, $\{\rho_i\}$  (which is not a subset of $\dom(F)$!) is bounded in $\mathcal{H}$. This again implies, by Prop.~\ref{prop:my}\ref{prop:my:4}, that $\|\nabla {}^\epsilon F^\lambda(\rho_i)\|$ is bounded, hence $\{v_{\text{eff},i}\}$ is bounded. Equation~\eqref{eq:grad-G-mono} combined with Eq.~\eqref{eq:step-length} gives
\[
t_i^{-1}\|\rho_{i+1} - \rho_i\|\|\Delta\rho_i\| = \|\Delta\rho_i\|^2 t_i^2 \geq \epsilon^2 \|\nabla{}^\epsilon F[\rho_i] + v_\text{ext} \|^2 = \epsilon^2 \|\nabla{}^\epsilon F^0[\rho_i] + v_{\text{eff},i} \|^2.
\]
Since $\|\Delta\rho_i\|t_i \to 0$, we now have
$\|\nabla{}^\epsilon F^0[\rho_i] + v_{\text{eff},i}\|\to 0$ and $\|\nabla{}^\epsilon F[\rho_i] + v_{\ext}\|\to 0$. The latter implies that $v_\ext = -\lim_{i\to+\infty} \nabla{}^\epsilon F[\rho_i]$. Let $\sigma\in\mathcal{H}$ be  an accumulation point of $\{\rho_i\}$, which is guaranteed to exist by the Bolzano--Weierstrass theorem. There is a convergent subsequence $\rho_{i_n}\to\sigma$ for which we have $v_\ext = -\lim_{n} \nabla{}^\epsilon F[\rho_{i_n}] = -\nabla{}^\epsilon F[\sigma]$. But then $\sigma \in \supdiff {}^\epsilon E[v_\ext]$ is a ground-state quasidensity of the exact regularized problem. We next obtain $\lim_{n\to+\infty} v_{\text{eff},i_n+1} = -\lim_{n\to+\infty}\nabla{}^\epsilon F^0[\rho_{i_n}] = - \nabla{}^\epsilon F^0[\sigma] = v_\text{ext} + {}^\epsilon v_{\text{Hxc}}[\sigma]$. Since the accumulation point was arbitrary, the sequence $\{\rho_i\}$ splits into subsequences converging to different ground-state densities $\sigma$ as claimed.
\end{proof}

\begin{remark}
    The step length choice in the proof of Theorem~\ref{thm:myksoda2} are not practical (if one can consider the MYKS scheme practical at all), since they require the whereabouts of the proximal point $\prox_{\epsilon F}\rho_i$. But knowledge of that would make Moreau--Yosida regularization approach redundant.
\end{remark}

\subsection{Density-functional approximations}

The preceding discussion so far has pertained to the \emph{exact} ground-state energy functional $E^\lambda \in \Gamma_0(\Hilb)$, which is an abstract setting far from actual numerical calculations. We now make some simple observations that opens up the possibility of a rigorous treatment of Kohn--Sham theory for model density functionals~\cite{Parr1989}, including the application and further analysis of the MYKSODA iteration scheme.

The Kohn--Sham decomposition relied on a family $\mathcal{U} = \{ E^\lambda \in -\Gamma_0(\Hilb) \mid \lambda\in[0,1]\}$. This family was such that $\supdiff E^{\lambda}[v] \neq \emptyset$ for any $v\in \Hilb$, and generates the family $\{ {}^\epsilon E^\lambda : \Hilb\to\R \mid \lambda\in[0,1] \}$ and the corresponding Moreau envelopes $\{ {}^\epsilon F^\lambda : \Hilb\to\R \mid \lambda\in[0,1] \}$. Moreover, the idea was present that solving for $\rho \in \supdiff {}^\epsilon E^0[v]$ for arbitrary $v\in\Hilb$, i.e., the noninteracting problem, was in some sense easy compared to solving for $\rho_\epsilon^1 \in \supdiff {}^\epsilon E^1[v_\text{ext}]$. Properties of the family $\mathcal{U}$ that were \emph{not} used in the Kohn--Sham decomposition was the concavity in the parameter $\lambda \in [0,1]$.

In Algorithm~\ref{algo:1}, the basic MYKS-SCF iteration scheme, the existence of solutions $\rho \in \supdiff {}^\epsilon E^0[v]$ for arbitrary $v\in\Hilb$ was essential, but otherwise no additional properties of $\mathcal{U}$ were used. In Algorithm~\ref{algo:2}, MYKSODA, and Theorem~\ref{thm:myksoda}, the same is true.

This motivates the introduction of an approximation, or model, ${}^\epsilon \tilde{E}_\text{Hxc}$ such that ${}^\epsilon \tilde{F}^1 = {}^\epsilon F^0 + {}^\epsilon \tilde{E}_\text{Hxc} \in \Gamma_0(\Hilb)$ is continuously differentiable, replacing the exact Moreau--Yosida regularization in Algorithm~\ref{algo:1} and \ref{algo:2}. Of course, \emph{how} to obtain such models is a different matter.

A second approach is to consider the \emph{exact} $\lambda$-dependent Hartree-exchange correlation functional $E^\lambda_\text{Hxc} = F^\lambda-F^0$, a functional with domain $\mathcal{D}^N$ but of otherwise unknown composition, and introduce a model functional $\tilde{E}^\lambda_{\text{Hxc}} : \mathcal{D}^N \to \R$ and corresponding model universal functionals, viz.,
\[ \tilde{F}^\lambda[\rho] := F^0[\rho] + \tilde{E}_\text{Hxc}^\lambda [\rho]. \]
This family may or may not be convex, but suppose for simplicity that it is, and that $\tilde{E}_{\text{Hxc}}^\lambda$ is in some sense ``cheap'' but also sufficiently regular as to keep $\tilde{F}^\lambda \in \Gamma_0(\Hilb)$. The family induces a model energy $\tilde{E}^\lambda \in -\Gamma_0(\Hilb)$ such that $\tilde{E}^0 = E^0$, retaining the ease of solution in the non-interacting limit.

The next step is to consider the Moreau--Yosida regularized model Hartree-exchange correlation energy,
\[ {}^\epsilon \tilde{E}_\text{Hxc}^\lambda[\rho] := {}^\epsilon \tilde{F}^\lambda[\rho] - {}^\epsilon F^0[\rho] = {}^\epsilon (F^0 + \tilde{E}^\lambda_\text{Hxc})[\rho] - {}^\epsilon F^0[\rho]. \]
The main problem is now to identify how and when ${}^\epsilon \tilde{E}^\lambda_{\text{Hxc}}$ can be easy to calculate given that it involves a non-trivial convolution, which may negate the ease of evaluation of $\tilde{E}^\lambda_\text{Hxc}[\rho]$.

\section{Conclusion}
\label{sec:conclusion}

In this chapter, we have reformulated exact DFT as introduced by Lieb in terms of the Moreau--Yosida regularization. Our starting point was the $N$-electron problem. Not considered were grand-canonical ensembles at finite or zero temperature, which adds some interesting aspects to the formalism~\cite{Parr1989}. Resulting in a compelling reformulation of DFT, including a rigorous formulation of Kohn--Sham theory, Moreau--Yosida regularization still leaves open the question of how to find functional approximations to the Hartree-exchange correlation energy.

The mathematical understanding of Moreau--Yosida regularized Kohn--Sham theory is in its early stages. For example, Theorem~\ref{thm:myksoda2} requires a finite-dimensional density space. Stronger and more general results are of course highly desired, and could turn Moreau--Yosida Kohn--Sham theory a practical tool when combined with approximate Hartree-exchange correlation functionals.


In current-density functional theory (CDFT), the ground-state energy $E[v,\mathbf{A}]$ of an $N$-electron system in the external potential $v$ and magnetic potential $\mathbf{A}$ is considered, leading to both the density $\rho$ and the current density $\mathbf{j} : \R^3\to\R^3$ being variables in the constrained-search functional~\cite{Vignale1987}. This requires a more general setting than that of a Hilbert space of densities. In Ref.~\cite{Laestadius2018b}, Moreau--Yosida regularized DFT was introduced in the abstract setting of a reflexive Banach space of densities, thereby accommodating CDFT and demonstrating that the regularization approach may also be important in other density-functional settings than the standard one considered in this chapter.

\bibliographystyle{unsrt}
\bibliography{refs.bib}

\end{document}